\documentclass{amsart}
\usepackage{graphicx}
\newtheorem{thm}{Theorem}[section]

\newtheorem{lem}[thm]{Lemma}
\newtheorem{prop}[thm]{Proposition}
\theoremstyle{definition}
\newtheorem{defn}[thm]{Definition}
\theoremstyle{remark} \theoremstyle{Proof}
\newtheorem{rem}[thm]{Remark}
\newtheorem{exple}[thm]{Example}
\numberwithin{equation}{section}

\author[O. Ajebbar]{Omar Ajebbar}
\address{Omar Ajebbar\\Department of Mathematics\\
Ibn Zohr University, Faculty of Sciences, Agadir\\
Morocco} \email{omar-ajb@hotmail.com}
\author[E. Elqorachi]{ Elhoucien Elqorachi}
\address{Elhoucien Elqorachi\\Department of Mathematics\\
Ibn Zohr University, Faculty of Sciences, Agadir\\
Morocco} \email{elqorachi@hotmail.com}
\begin{document}

\title[Extensions of the Cosine-Sine functional equation]{Extensions of the Cosine-Sine functional equation}
\keywords{Functional equation; Cosine-Sine functional equation;
Levi-Civita; Monoid; Group; Additive function; Multiplicative
function; Character.}

\thanks{2010 Mathematics Subject
Classification. Primary 39B52; Secondary 39B32}

\begin{abstract} The aim of the present paper is to give
extensions of the cosine-sine functional equation.
\end{abstract}
\maketitle
\section{Introduction}
\subsection{Connection to the literature}\quad\quad\quad\\
In the present paper we solve the following functional equations:
\begin{equation}\label{eq51}f(xy)=g_{1}(x)h_{1}(y)+g(x)h_{2}(y)+\chi(x)h(y),\,x,y\in
G,\end{equation} on a group $G$, where $g$ is the average of two
different characters $\mu_{1}$ and $\mu_{2}$ of $G$ and $\chi$ is a
character of $G$ such that $\mu_{1}\neq\chi$ and $\mu_{2}\neq\chi$;
\begin{equation}\label{Eq5-1}f(xy)=g(x)h(y)+\sum_{j=1}^{N}\chi_j(x)h_j(y),\,x,y\in
G,\end{equation} where $\chi_{1},\cdot\cdot\cdot,\chi_{N}$ are $N$
distinct characters of a group $G$;
\begin{equation}\label{equ51}f(xy)=g_{1}(x)h_{1}(y)+g(x)h_{2}(y)+\chi(x)A(x)h(y),\quad
x,y\in G,\end{equation} where $g$ is the average of two given
nonzero multiplicative functions $\mu$ and $\chi$ on a monoid $G$
and $A$ is an additive function on $G$ such that $\chi A\neq0$;
\begin{equation}\label{Eq.5-1}f(xy)=g_{1}(x)h_{1}(y)+\mu(x)h_{2}(y)+\chi(x)A(x)h(y),\,x,y\in
G,\end{equation} where $\mu$ and $\chi$ are two different characters
of a group $G$ and $A:G\rightarrow\mathbb{C}$ is a nonzero additive
function on $G$. To solve the functional equations (\ref{eq51}),
(\ref{Eq5-1}), (\ref{equ51}) and (\ref{Eq.5-1}) we will need to
solve respectively the corresponding functional equations
\begin{equation}\label{eq52}f(xy)=f(x)h_{1}(y)+g(x)h_{2}(y)+\chi(x)h(y),\,x,y\in G,\end{equation}
\begin{equation}\label{Eq5-2}f(xy)=f(x)h(y)+\sum_{j=1}^{N}\chi_j(x)h_j(y),\,x,y\in G,\end{equation}
\begin{equation}\label{equ52}f(xy)=f(x)h_{1}(y)+g(x)h_{2}(y)+\chi(x)A(x)h(y),\,x,y\in G,\end{equation}
and
\begin{equation}\label{Eq.5-2}f(xy)=f(x)h_{1}(y)+\mu(x)h_{2}(y)+\chi(x)A(x)h(y),\,x,y\in G.\end{equation}
We notice that (\ref{eq51}), (\ref{Eq5-1}), (\ref{equ51}),
(\ref{Eq.5-1}), (\ref{eq52}), (\ref{Eq5-2}), (\ref{equ52}) and
(\ref{Eq.5-2}) are examples of Levi-Civita's functional equation
$$f(xy)=\sum_{j=1}^{N}g_j(x)h_j(y),\; x,y \in G,$$ on which lots of work
has been done. Sz\'ekelyhidi \cite[Section 10]{Sz�kelyhidi}
computes for any abelian group $(G,+)$ the solution of
$$f(x+y)=\sum_{j=1}^{N}g_j(x)h_j(y),\; x,y \in G,$$ under the
assumption that each of the sets
$\{g_{1},g_{2},\cdot\cdot\cdot,g_{N}\}$ and
$\{h_{1},h_{2},\cdot\cdot\cdot,h_{N}\}$ is linearly independent. The
structure of any solution of Levi-Civita's functional equation on
monoids was described by using matrix-coefficients of the right
regular representations \cite[Theorem 5.2]{Stetkaer1}. In \cite[pp .
25-27]{Andreescu et al.} Vincze's equation, which is a particular
case of our functional equations, was solved. We refer also to
\cite{Acz�l, LC, SH, Sz�kelyhidi1}.
\par If $h=0$ then the functional equations
(\ref{eq51}) and (\ref{equ51}) become
$$f(xy)=g_{1}(x)h_{1}(y)+g(x)h_{2}(y),$$ which was recently solved by
Ebanks \cite{EB}, Stetk\ae r \cite{Stetkaer2}, and Ebanks and
Stetk{\ae}r \cite{Ebanks and Stetkaer}.
\par If $h=0$, $f=g_{1}=h_{2}$ and $h_{1}=g$ then Eq. (\ref{eq51}) and (\ref{equ51}) reduce to the sine addition law
$$f(xy)=f(x)g(y)+g(x)f(y)$$ for which solutions are known on semigroups, see for example \cite[Theorem
4.1]{Stetkaer1}.\\
Chung, Kannappan and Ng \cite{45} solved the  functional equation
$$f(xy)=f(x)g(y)+f(y)g(x)+h(x)h(y),\; x,y \in G,$$ where $G$ is a
group. Recently, the results of \cite{45} were extended by the
authors \cite{Ajebbar and Elqorachi} to semigroups generated by
their squares.
\par We relate the solutions of the functional equations
(\ref{eq51}), (\ref{eq52}), (\ref{Eq5-1}), (\ref{Eq5-2}),
(\ref{Eq.5-1}) and (\ref{Eq.5-2})  to those of a functional equation
of type
$$f(xy)=f(x)\chi_{1}(y)+\chi_{2}(x)f(y),$$ where $\chi_{1}$ and
$\chi_{2}$ are two characters of $G$, which was solved by
Stetk{\ae}r \cite{Stetkaer2}.
\subsection{Motivation}\quad\quad\quad\\
\par In \cite{EB} Ebanks solved the functional equation
\begin{equation}\label{Eq5-2-1}f(xy)=k(x)M(y)+g(x)h(y),\,x,y\in S,\end{equation}
for four unknown central functions $f,g,h,k$ on certain semigroups
$S$, where $M$ is a fixed multiplicative function on $S$.
\par In \cite{Stetkaer2} Stetk{\ae}r solved, by removing the restriction of the solutions, the functional equation
\begin{equation}\label{Eq5-2-2}f(xy)=f(x)\chi_{1}(y)+\chi_{2}(x)f(y),\,x,y\in G\end{equation}
and its Pexiderized version
\begin{equation}\label{Eq5-2-3}f(xy)=g_{1}(x)h_{1}(y)+\chi_{2}(x)h_{2}(y),\,x,y\in G\end{equation}
on a group $G$, where $\chi_{1}$ and $\chi_{2}$ are fixed characters
of $G$.
\par In \cite{Ebanks and Stetkaer} Ebanks and Stetk{\ae}r studied the more general variant of (\ref{Eq5-2-2})
\begin{equation}\label{Eq5-2-4}f(xy)=g_{1}(x)h_{1}(y)+g(x)h_{2}(y),\end{equation}
for four unknown $f,g_{1},h_{1},h_{2}$ on a monoid, where $g$ is a
linear combination of $n\geq2$ distinct multiplicative functions.
\par Recently Belfakih and Elqorachi \cite[Theorem 2.7]{BEL and ELQ } solved the functional equation $$f(xy)=g(x)h(y)+\sum_{j=1}^{n}g_{j}(x)h_{j}(y),\,x,y\in M$$
on a monoid $M$, where $g_{j}$ are linear combinations of at least
$2$  distinct nonzero multiplicative functions on $M$, and
 $f,g,h,h_{j}$ are the unknown functions.
\par Let $G$ be a group. Let $\mu$ and $\chi$ be two
different characters of $G$ and let $A:G\rightarrow\mathbb{C}$ be a
non zero additive function. Let $c_{1},c_{2}\in\mathbb{C}$ be two
constants satisfying $1+c_{1}c_{2}^{2}=0$. We check by elementary
computations that the functions $f,g,k:G\rightarrow\mathbb{C}$
defined by $f:=-c_{1}\mu+c_{1}\chi-c_{1}c_{2}\chi\,A$,
$g:=\dfrac{\mu+\chi}{2}-\dfrac{c_{2}}{2}\chi\,A$ and $k:=\chi\,A$
satisfy the functional equation $$f(xy)=f(x)g(y)+g(x)f(y)+k(x)k(y)$$
on the group $G$, which implies, by simple computations, that the
triple $(f,g,h)$ is a solution of the functional equation
$$f(xy)=f(x)g(y)+\dfrac{\mu(x)+\chi(x)}{2}f(y)+\chi(x)A(x)h(y)$$ for
all $x,y\in G$, where $h:=-\dfrac{1}{2}c_{2}f+\chi\,A$ and
$h\neq\chi\,A$ because $c_{1},c_{2}\in\mathbb{C}\setminus\{0\}$.
Moreover, for a given character $m$ of $G$, we check easily that the
quadruple $(m,\dfrac{m+\mu+\chi}{2}+\chi\,A,2m,-2m,-2m)$ satisfies
the functional equation
$$f(xy)=g_{1}(x)h_{1}(y)+\dfrac{\mu(x)+\chi(x)}{2}h_{2}(y)+\chi(x)A(x)h(y)$$
while the functions $f:=\mu$, $g_{1}:=-\mu+\chi\,A$, $h_{1}:=\chi$,
$h_{2}:=\mu+\chi$ and $h:=-\chi$ satisfy the functional equation
$$f(xy)=g_{1}(x)h_{1}(y)+\mu(x)h_{2}(y)+\chi(x)A(x)h(y)$$ with $g_{1}\neq f$, $h_{2}\neq f$, $h\neq0$ and $h\neq\chi\,A$.
\par The papers cited above motivate us to deal with the functional
equations (\ref{eq51}), (\ref{Eq5-1}), (\ref{equ51}),
(\ref{Eq.5-1}), (\ref{eq52}), (\ref{Eq5-2}), (\ref{equ52}) and
(\ref{Eq.5-2}), which are not covered by the cited works.
\subsection{About our results}\quad\quad\quad\\
\par In the present paper we solve the functional equations (\ref{eq51}), (\ref{Eq5-1}),
(\ref{equ51}), (\ref{Eq.5-1}), (\ref{eq52}), (\ref{Eq5-2}),
(\ref{equ52}) and (\ref{Eq.5-2}), which are not in the literature.
We express the solutions of the functional equations (\ref{eq51}),
(\ref{Eq5-1}), (\ref{equ51}) and (\ref{Eq.5-1}) respectively in term
of the solutions of (\ref{eq52}), (\ref{Eq5-2}), (\ref{equ52}) and
(\ref{Eq.5-2}). We relate some categories of solutions to the
solutions of the functional equation
$f(xy)=f(x)\chi_{1}(y)+\chi_{2}(x)f(y)$ where $\chi_{1}$ and
$\chi_{2}$ are two characters of the group $G$, which was recently
solved by Stetk{\ae}r \cite{Stetkaer2}. Our results are extensions
of those of \cite{Stetkaer2} and \cite{Ebanks and Stetkaer}.\\
There is no connection between the functional equations
(\ref{eq51}), (\ref{Eq5-1}), (\ref{equ51}) and (\ref{Eq.5-1}) nor
between (\ref{eq52}), (\ref{Eq5-2}), (\ref{equ52}) and
(\ref{Eq.5-2}), but the reasonings for the functional equations with
additive function , i.e. (\ref{equ51}), (\ref{Eq.5-1}),
(\ref{equ52}) and (\ref{Eq.5-2}) are similar to those for the
functional equations that not contain the additive function, i.e.
(\ref{eq51}), (\ref{Eq5-1}), (\ref{eq52}), (\ref{Eq5-2}). This is
due to fact that the resolution of this functional equations\\
(1) make use the linear independence between some unknown functions
and the given arguments in equations, namely multiplicative
functions, additive function or characters.\\(2) relate some of them
to the functional equation $f(xy)=f(x)\chi_{1}(y)+\chi_{2}(x)f(y)$,
where $\chi_{1}$ and $\chi_{2}$ are two characters of the group $G$.
\par The organization of the paper is as follows. In the next
section we give some definitions and notations. In the third section
we give preliminary results that we need in the paper. In sections
4, 5, 6 and 7 we prove our main results and give
some examples.\\
Our methods are elementary algebraic manipulations (no analysis or
geometry come into play).
\section{Set Up and Terminology}
\begin{defn} Let $f:S\to\mathbb{C}$ be a function on a
semigroup $S$. We say that\\
$f$ is central if $f(xy)=f(yx)$ for all $x,y\in G$.\\
$f$ is additive if $f(xy)=f(x)+f(y)$ for all $x,y\in S$.\\
$f$ is multiplicative if $f(xy)=f(x)f(y)$ for all $x,y\in S$.\\
If $S=G$ is a group, then we say that $f$ is a character of $G$ if
$f$ is a homomorphism of $G$ into $\mathbb{C}^{*}$, where
$\mathbb{C}^{*}$ denotes the multiplicative group of nonzero complex
numbers.
\end{defn}
If $G$ is a group, then $[G,G]$ denotes the commutator of $G$, i.e.,
the subgroup of $G$ generated by the set of commutators
$[x,y]:=xyx^{-1}y^{-1},\, x,y\in G$.\\
If $\chi_{1}$ and $\chi_{2}$ are two characters of a group $G$, we
denote by $\varphi_{\chi_{1},\chi_{2}}$ a solution of the functional
equation
$$f(xy)=f(x)\chi_{1}(y)+\chi_{2}(x)f(y),\,x,y\in G.$$
\textbf{Blanket assumption:} Throughout this paper $G$ denotes a
group or a monoid with identity element $e$.
\section{Preliminaries}
In this section we give the essential tools that we will need in the
paper.
\begin{prop}\label{prop51} \cite[Theorem 3.18(a)]{Stetkaer1} Let $S$ be a semigroup and
$n\in\mathbb{N}$. Let
$\chi,\chi_{2},\cdot\cdot\cdot,\chi_{n}:S\to\mathbb{C}$ be $n$
different multiplicative functions and let
$a_{1},a_{2},\cdot\cdot\cdot,a_{n}\in\mathbb{C}$.\\
If $a_{1}\chi+a_{2}\chi_{2}+\cdot\cdot\cdot+a_{n}\chi_{n}=0$, then
$a_{1}\chi=a_{2}\chi_{2}=\cdot\cdot\cdot=a_{n}\chi_{n}=0$, in
particular, any set of non zero multiplicative functions is linearly
independent.
\end{prop}
\begin{lem}\label{lem52} \cite[pp. 1122-1123]{Ajebbar and Elqorachi} Let $a:S\rightarrow\mathbb{C}$ be an additive
function and $m:S\rightarrow\mathbb{C}$ be a multiplicative function
on a semigroup $S$.
\par If $ma=\Sigma_{j=1}^{N}c_{j}\mu_{j}$, where
$c_{j}\in\mathbb{C}$ and $\mu_{j}:S\rightarrow\mathbb{C}$ is
multiplicative for each $j=1,2,\cdot\cdot\cdot{},N$, then $ma=0$.
\end{lem}
\par  Let $\chi_{1}$ and $\chi_{2}$ be characters
of $G$. The following result, which is taken from \cite[Proposition
4, Theorem 11]{Stetkaer2}, solves the functional equation
\begin{equation}\label{Eq}f(xy)=f(x)\chi_{1}(y)+\chi_{2}(x)f(y),\,x,y\in G.\end{equation}
\begin{thm}\label{Thm53} \cite[p.5, pp .8-9]{Stetkaer2} Let $G$ be a group.\\
(1) If $\chi_{1}=\chi_{2}$ then the solutions of (\ref{Eq}) are the
functions $f=\chi_{1}\,A$ where $A:G\to\mathbb{C}$ is
an additive function on $G$.\\
(2) If $\chi_{1}\neq\chi_{2}$ then the solutions of (\ref{Eq}) are
the functions
$$f(x)=\alpha(\chi_{1}(x)-\chi_{2}(x))+A([y_{0},x])\chi_{1}(x),\,x\in
G,$$ where $y_{0}\in G$ such that
$\chi_{1}(y_{0})\neq\chi_{2}(y_{0})$, $\alpha$ ranges over
$\mathbb{C}$ and $A:[G,G]\to\mathbb{C}$ ranges over the additive
functions on $[G,G]$ with the transformation property
$$A(xcx^{-1})=\dfrac{\chi_{2}(x)}{\chi_{1}(x)}A(c)\quad\text{for all}\quad x\in G\quad\text{and}\quad c\in [G,G].$$
\end{thm}
\begin{rem} According to \cite[Proposition 4]{Stetkaer2}, if
$\chi_{1}\neq\chi_{2}$ the central solutions of (\ref{Eq}) are the
functions $f=\alpha(\chi_{1}-\chi_{2})$ where $\alpha$ ranges over
$\mathbb{C}$.
\end{rem}
\begin{lem}\label{lem53} Let $G$ be a group. Let $(f,g_{1},h_{1},h_{2},h)$ be a solution
of the functional equation (\ref{eq51}).\\
(1) If $h_{1}=0$ then $f=h=b\chi$, $h_{2}=0$ and $g_{1}$ is arbitrary, where $b\in\mathbb{C}$ is a constant.\\
(2) If $h_{1}\neq0$ and the set $\{f,\mu_{1},\mu_{2},\chi\}$ is
linearly dependent then the set $\{g_{1},\mu_{1},\mu_{2},\chi\}$ is
also linearly dependent.
\end{lem}
\begin{proof} (1) Assume that $h_{1}=0$. Then $g_{1}$ is arbitrary. By putting $y=e$ in
(\ref{eq51}), we get that $f=a\mu_{1}+a\mu_{2}+b\chi$ where
$a:=\dfrac{h_{2}(e)}{2}$ and $b:=h(e)$. Substituting this back into
(\ref{eq51}) we obtain
\begin{equation*}\begin{split}&a\mu_{1}(y)\mu_{1}(x)+a\mu_{2}(y)\mu_{2}(x)+b\chi(y)\chi(x)\\
&=\dfrac{h_{2}(y)}{2}\mu_{1}(x)+\dfrac{h_{2}(y)}{2}\mu_{2}(x)+h(y)\chi(x),\end{split}\end{equation*}
for all $x,y\in G$. Since the characters $\mu_{1}$, $\mu_{2}$ and
$\chi$ are different the set $\{\mu_{1},\mu_{2},\chi\}$ is linearly
independent. Then we get, from the identity above, that
$\dfrac{h_{2}}{2}=a\mu_{1}=a\mu_{2}$ and $h=b\chi$. So that $a=0$
because $\mu_{1}\neq\mu_{2}$.
Hence, $f=h=b\chi$ and $h_{2}=0$.\\
(2) Assume that $h_{1}\neq0$ and $\{f,\mu_{1},\mu_{2},\chi\}$ is
linearly dependent. Since the set $\{\mu_{1},\mu_{2},\chi\}$ is
linearly independent there exist
$\alpha_{1},\alpha_{2},\alpha_{3}\in\mathbb{C}$ such that
$f=\alpha_{1}\mu_{1}+\alpha_{2}\mu_{2}+\alpha_{3}\chi$. So that
$$f(xy)=\alpha_{1}\mu_{1}(x)\mu_{1}(y)+\alpha_{2}\mu_{2}(x)\mu_{2}(y)+\alpha_{3}\chi(x)\chi(y)$$
for all $x,y\in G$, from which we get, taking (\ref{eq51}) into
account, that
\begin{equation*}\begin{split}&g_{1}(x)h_{1}(y)=[\alpha_{1}\mu_{1}(y)-\dfrac{1}{2}h_{2}(y)]\mu_{1}(x)+[\alpha_{2}\mu_{2}(y)-\dfrac{1}{2}h_{2}(y)]\mu_{2}(x)\\
&+[\alpha_{3}\chi(y)-h(y)]\chi(x)\end{split}\end{equation*} for all
$x,y\in G$. As $h_{1}\neq0$ we deduce from the identity above, by
choosing $y_{0}\in G$ such that $h_{1}(y_{0})\neq0$, that the set
$\{g_{1},\mu_{1},\mu_{2},\chi\}$ is linearly dependent.
\end{proof}
\begin{prop}\label{prop54} Let $G$ be a monoid. For $f=0$ the solutions
of the functional equation (\ref{equ51}) can be listed as follows\\
(1) $g_{1}$ arbitrary and $h_{1}=h_{2}=h=0$;\\
(2) $g_{1}=a\dfrac{\mu+\chi}{2}+b\,\chi\,A$, $h_{1}\neq0$ arbitrary,
$h_{2}=-ah_{1}$ and $h=-bh_{1}$, where $a$ and $b$ range over
$\mathbb{C}$.
\end{prop}
\begin{proof} (1) Assume that $h_{1}=0$. Then $g_{1}$ is arbitrary. Let $y\in G$ be arbitrary.
Since $f=0$ the functional equation (\ref{equ51}) implies
$$\dfrac{h_{2}(y)}{2}\mu(x)+\dfrac{h_{2}(y)}{2}\chi(x)=-h(y)\chi(x)A(x)$$ for all $x\in G$.
So, the set
$\{\mu,\chi,\chi\,A\}$ being linearly independent, $h_{2}(y)=h(y)=0$ and $y$ arbitrary, we deduce that $h_{2}=h=0$. \\
(2) Assume that $h_{1}\neq0$. Taking $x=e$ in (\ref{equ51}) and
seeing that $\mu(e)=\chi(e)=1$ and $A(e)=0$ we get that
$g_{1}(e)h_{1}(y)+h_{2}(y)=0$ for all $y\in G$. So that
\begin{equation}\label{equ53}h_{2}=-ah_{1}\end{equation} where $a:=g_{1}(e)$. Substituting (\ref{equ53}) back into
(\ref{equ51}) we get that
$$(g_{1}(x)-a\dfrac{\mu(x)+\chi(x)}{2})h_{1}(y)+h(y)\chi(x)A(x)=0$$
for all $x,y\in G$. Choosing $y_{0}\in G$ such that
$h_{1}(y_{0})\neq0$ we deduce from the last functional equation that
\begin{equation}\label{equ53}g_{1}=a\dfrac{\mu+\chi}{2}+b\,\chi\,A\end{equation} where
$b:=-\dfrac{h(y_{0})}{h_{1}(y_{0})}$.\\
Substituting (\ref{equ53}) and (\ref{equ53}) back into (\ref{equ51})
we get that
$$(bh_{1}(y)+h(y))\chi(x)A(x)=0$$ for all $x,y\in G$. Since
$\chi\,A\neq0$ we derive that $$h=-bh_{1}.$$ Conversely, it is easy
to check that the formulas of (1) and (2) define solutions of
(\ref{equ51}) for $f=0$.
\end{proof}
\begin{prop}\label{prop55} Let $G$ be a group and let $\alpha\in\mathbb{C}$ be a constant. For $f=\alpha\mu$ the solutions
of the functional equation (\ref{Eq.5-1}) can be listed as follows\\
(1) $g_{1}$ arbitrary, $h_{1}=0$, $h_{2}=\alpha\mu$ and $h=0$;\\
(2) $g_{1}=a\mu+b\chi\,A$, $h_{1}\neq0$ arbitrary,
$h_{2}=-ah_{1}+\alpha\mu$ and $h=-bh_{1}$, where $a$ and $b$ range
over $\mathbb{C}$.
\end{prop}
\begin{proof} By proceeding as in the proof of Proposition
\ref{prop54}.
\end{proof}
\section{Solutions of Eq. (\ref{eq52}) and Eq. (\ref{eq51})}
In this section $\mu_{1},\,\mu_{2}$ and $\chi$ are different
characters of a group $G$, and $g:=\dfrac{\mu_{1}+\mu_{2}}{2}$.
\subsection{Solutions of Eq. (\ref{eq52})}\quad\\\\
In this subsection we solve the functional equation (\ref{eq52}),
i.e.,
$$f(xy)=f(x)h_{1}(y)+g(x)h_{2}(y)+\chi(x)h(y),\,x,y\in G,$$
where $f,h_{1},h_{2},h:G\to\mathbb{C}$ are the unknown functions to
be determined.
\par Proposition \ref{prop56} solves the functional equation (\ref{eq52}) when the set $\{f,\mu_{1},\mu_{2},\chi\}$ is
linearly independent.
\begin{prop}\label{prop56} Let $(f,h_{1},h_{2},h)$ be a solution of
(\ref{eq52}) such that the set $\{f,\mu_{1},\mu_{2},\chi\}$ is
linearly independent. Then
$$f=c\mu+\varphi_{\mu,\chi},\,h_{1}=\mu,\,h_{2}=0\,\,\text{and}\,\,h=\varphi_{\mu,\chi},$$
where $\mu$ is a character of $G$ and $c\in\mathbb{C}$ is a
constant.
\end{prop}
\begin{proof}
Let $x,y,z\in G$ be arbitrary. We compute $f(xyz)$ as $f(x(yz))$ and
then as $f((xy)z)$. Using Eq. (\ref{eq52}) and that
$g=\dfrac{\mu_{1}+\mu_{2}}{2}$ we obtain
\begin{equation}\label{eq54}f(x(yz))=f(x)h_{1}(yz)+\dfrac{1}{2}\mu_{1}(x)h_{2}(yz)+\dfrac{1}{2}\mu_{2}(x)h_{2}(yz)+\chi(x)h(yz).\end{equation}
On the other hand, by applying Eq. (\ref{eq52}) to the pair $(xy,z)$
we get that
\begin{equation*}\begin{split}f((xy)z)=f(xy)h_{1}(z)+g(xy)h_{2}(z)+\chi(xy)h(z)\quad\quad\quad\quad\quad\quad\quad\quad\quad\quad\quad\quad\quad\quad\quad\\
=[f(x)h_{1}(y)+\dfrac{1}{2}\mu_{1}(x)h_{2}(y)+\dfrac{1}{2}\mu_{2}(x)h_{2}(y)+\chi(x)h(y)]h_{1}(z)+\dfrac{1}{2}\mu_{1}(x)\mu_{1}(y)h_{2}(z)\\
+\dfrac{1}{2}\mu_{2}(x)\mu_{2}(y)h_{2}(z)+\chi(x)\chi(y)h(z).\quad\quad\quad\quad\quad\quad\quad\quad\quad\quad\quad
\quad\quad\quad\quad\quad\quad\quad\quad\quad\end{split}\end{equation*}
So that
\begin{equation}\label{eq55}\begin{split}f((xy)z)&=f(x)h_{1}(y)h_{1}(z)+\dfrac{1}{2}\mu_{1}(x)[h_{2}(y)h_{1}(z)+\mu_{1}(y)h_{2}(z)]\\
&+\dfrac{1}{2}\mu_{2}(x)[h_{2}(y)h_{1}(z)+\mu_{2}(y)h_{2}(z)]\\&+\chi(x)[h(y)h_{1}(z)+\chi(y)h(z)].\end{split}\end{equation}
As the set $\{f,\mu_{1},\mu_{2},\chi\}$ is linearly independent and
$x,y,z$ are arbitrary, we derive from (\ref{eq54}) and (\ref{eq55})
that $(h_{1},h_{2},h)$ is a solution of the following functional
equations as a system
\begin{equation}\label{eq56}h_{1}(xy)=h_{1}(x)h_{1}(y),\end{equation}
\begin{equation}\label{eq57}h_{2}(xy)=h_{2}(x)h_{1}(y)+\mu_{1}(x)h_{2}(y),\end{equation}
\begin{equation}\label{eq58}h_{2}(xy)=h_{2}(x)h_{1}(y)+\mu_{2}(x)h_{2}(y)\end{equation}
and
\begin{equation}\label{eq58-1}h(xy)=h(x)h_{1}(y)+\chi(x)h(y)\end{equation}
for all $x,y\in G$. The functional equation (\ref{eq56}) implies
that $h_{1}=\mu$ where $\mu:G\to\mathbb{C}$ is a multiplicative
function. By subtracting (\ref{eq57}) from (\ref{eq58}) we get that
$$[\mu_{1}(x)-\mu_{2}(x)]h_{2}(y)=0$$ for all $x,y\in G$. Since $\mu_{1}\neq\mu_{2}$ we have $h_{2}(y)=0$. So, $y$ being
arbitrary, we deduce that $h_{2}=0$.
\par If $\mu=0$, then $h_{1}=0$ and the
functional equation (\ref{eq58-1}) reduces to
$$h(xy)=\chi(x)h(y)$$ for all $x,y\in G$, which yields, by putting
$y=e$, that $h=\alpha\chi\,\,\text{where}\,\,\alpha:=h(e)$. Thus the
functional equation (\ref{eq52}) becomes $f(xy)=\alpha\chi(xy)$ for
all $x,y\in G$. This implies, by putting $y=e$, that $f=\alpha\chi$,
which contradicts the linear independence of the set
$\{f,\mu_{1},\mu_{2},\chi\}$.
\par Hence $\mu\neq0$ and then $\mu$ is a character of $G$. So, the
functional equation (\ref{eq58-1}) implies that
$$h=\varphi_{\mu,\chi}.$$ So that the functional equation
(\ref{eq52}) reduces to
$$f(xy)=f(x)\mu(y)+\chi(x)\varphi_{\mu,\chi}(y)$$
for all $x,y\in G$, from which we get, by putting $x=e$ and
$c=f(e)$, that $f=c\mu+\varphi_{\mu,\chi}$.
\end{proof}
\begin{thm}\label{thm56} The solutions $(f,h_{1},h_{2},h)$ of
(\ref{eq52}) are:\\
(1) $f=c\mu+\varphi_{\mu,\chi},\,h_{1}=\mu,\,h_{2}=0$ and
$h=\varphi_{\mu,\chi}$, where $\mu$ is a character of $G$ and
$c\in\mathbb{C}$ is a constant;\\
(2) $f=c\chi,\,h_{1}\,\,\text{is arbitrary},\,h_{2}=0\,\,\text{and}\,\,h=c(\chi-h_{1})$, where $c\in\mathbb{C}$ is a constant;\\
(3)$$f=a\mu_{1}+b\mu_{2}+c\chi,\,\,h_{1}=\dfrac{1}{a-b}(a\mu_{1}-b\mu_{2}),$$
$$h_{2}=-\dfrac{2ab}{a-b}(\mu_{1}-\mu_{2}),\,\,h=\dfrac{c}{a-b}(-a\mu_{1}+b\mu_{2}+(a-b)\chi),$$
where $a,b,c\in\mathbb{C}$ are constants such that $a\neq b$.
\end{thm}
\begin{proof} It is easy to check that the formulas in parts (1)-(3) define solutions of the
functional equation (\ref{eq52}), so left is to show that any
solution $(f,h_{1},h_{2},h)$ fits into (1)-(3).
\par If the set $\{f,\mu_{1},\mu_{2},\chi\}$ is linearly independent
we get, by Proposition \ref{prop56}, the part (1) of Theorem
\ref{thm56}.
\par In the remainder of the proof we assume that $\{f,\mu_{1},\mu_{2},\chi\}$ is linearly dependent.
Recall that the set $\{\mu_{1},\mu_{2},\chi\}$ is linearly
independent because $\mu_{1}$, $\mu_{2}$ and $\chi$ are different
characters of $G$. Then there exists a triple
$(a,b,c)\in\mathbb{C}^{3}$ such that
\begin{equation}\label{eq59-1}f=a\mu_{1}+b\mu_{2}+c\chi.\end{equation}
Substituting this back into (\ref{eq52}) we get by a small
computation that
\begin{equation*}\begin{split}&a\mu_{1}(y)\mu_{1}(x)+b\mu_{2}(y)\mu_{2}(x)+c\chi(y)\chi(x)\\
&=[ah_{1}(y)+\dfrac{1}{2}h_{2}(y)]\mu_{1}(x)+[bh_{1}(y)+\dfrac{1}{2}h_{2}(y)]\mu_{2}(x)+[ch_{1}(y)+h(y)]\chi(x),\end{split}\end{equation*}
for all $x,y\in G$. Since $\mu_{1}$, $\mu_{2}$ and $\chi$ are
different characters of $G$, we derive by Proposition \ref{prop51}
that $(h_{1},h_{2},h)$ is a solution of the following identities as
a system
\begin{equation}\label{eq59}ah_{1}+\dfrac{1}{2}h_{2}=a\mu_{1},\end{equation}
\begin{equation}\label{eq510}bh_{1}+\dfrac{1}{2}h_{2}=b\mu_{2}\end{equation}
and
\begin{equation}\label{eq511}ch_{1}+h=c\chi.\end{equation}
When we subtract (\ref{eq510}) from (\ref{eq59}) we obtain
\begin{equation}\label{eq512}(a-b)h_{1}=a\mu_{1}-b\mu_{2}.\end{equation}
\par If $a=b$ the identity (\ref{eq512})
reduces to $a(\mu_{1}-\mu_{2})=0$. Hence $a=b=0$ because
$\mu_{1}\neq\mu_{2}$. So, taking (\ref{eq59-1}) and (\ref{eq511})
into account, we get that
$$f=c\chi,\,h_{1}\,\,\text{is
arbitrary},\,h_{2}=0\,\,\text{and}\,\,h=c(\chi-h_{1}),$$ which is
solution (2).
\par If $a\neq b$ the identity (\ref{eq512}) implies that
$$h_{1}=\dfrac{1}{a-b}(a\mu_{1}-b\mu_{2}).$$
Substituting this back into (\ref{eq59}) and (\ref{eq511}) we derive
that
$$h_{2}=2a(\mu-h_{1})=2a[\mu_{1}-\dfrac{1}{a-b}(a\mu_{1}-b\mu_{2})]=-\dfrac{2ab}{a-b}(\mu_{1}-\mu_{2})$$
and
$$h=c(\chi-h_{1})=c[\chi-\dfrac{1}{a-b}(a\mu_{1}-b\mu_{2})]=\dfrac{c}{a-b}(-a\mu_{1}+b\mu_{2}+(a-b)\chi).$$
The solution occurs in part (3). This completes the proof.
\end{proof}
\subsection{Solutions of Eq. (\ref{eq51})}\quad\\\\
In Theorem \ref{Thm57} we solve the functional equation
(\ref{eq51}), i.e.,
$$f(xy)=g_{1}(x)h_{1}(y)+g(x)h_{2}(y)+\chi(x)h(y),\quad x,y\in G,$$ where
$f,g_{1},h_{1},h_{2},h:G\to\mathbb{C}$ are the unknown functions to
be determined.
\begin{thm}\label{Thm57} The solutions $(f,g_{1},h_{1},h_{2},h)$ of
(\ref{eq51}) can be listed as follows:\\
(1) $f=b\chi$, $g_{1}$ is arbitrary, $h_{1}=h_{2}=0$ and $h=b\chi$, where $b\in\mathbb{C}$ is a constant;\\
(2) $f=b\chi$, $g_{1}=\alpha\mu_{1}+\alpha\mu_{2}+\gamma\chi$,
$h_{1}\neq0$ is arbitrary, $h_{2}=-2\alpha h_{1}$ and
$h=b\chi-\gamma h_{1}$, where $b\in\mathbb{C}\setminus\{0\}$ and $\alpha,\gamma\in\mathbb{C}$ are constants;\\
(3)
$$f=(d_{1}-d_{2})(\dfrac{1}{2}a_{1}\mu_{1}-\dfrac{1}{2}a_{2}\mu_{2}-\dfrac{1}{2}a_{3}\chi),\,
g_{1}=\dfrac{1}{2}d_{1}\mu_{1}+\dfrac{1}{2}d_{2}\mu_{2}+d_{3}\chi,$$
$$h_{1}=a_{1}\mu_{1}+a_{2}\mu_{2},\,\,h_{2}=-a_{1}d_{2}\mu_{1}-a_{2}d_{1}\mu_{2}
\,\,\text{and}\,\,
h=-a_{1}d_{3}\mu_{1}-a_{2}d_{3}\mu_{2}-\dfrac{1}{2}a_{3}(d_{1}-d_{2})\chi,$$
where $a_{1},a_{2},a_{3},d_{1},d_{2},d_{3}\in\mathbb{C}$ are
constants such that $d_{1}\neq d_{2}$;\\
(4) $f=a_{1}b\mu+b\varphi_{\mu,\chi}$,
$g_{1}=-\frac{1}{2}a_{2}\mu_{1}-\frac{1}{2}a_{2}\mu_{2}-a_{3}\chi+a_{1}\mu+\varphi_{\mu,\chi}$,
$h_{1}=b\mu$, $h_{2}=a_{2}b\mu$ and
$h=a_{3}b\mu+b\varphi_{\mu,\chi}$, where
$b\in\mathbb{C}\setminus\{0\},a_{1},a_{2},a_{3}\in\mathbb{C}$ are
constants and $\mu$ is a character of $G$;\\
(5)
$$f=a_{1}b\chi,\,\,g_{1}=-\frac{1}{2}a_{1}\mu_{1}-\frac{1}{2}a_{2}\mu_{2}+(a_{1}-a_{2})\chi,\,\,h_{1}\,\,\text{is
arbitrary},$$
$$h_{2}=a_{2}h_{1}\,\,\text{and}\,\,h=a_{1}b\chi-(a_{1}-a_{2})h_{1},$$
where $b\in\mathbb{C}\setminus\{0\}$ and
$a_{1},\,a_{2},\,a_{3}\in\mathbb{C}$ are constants.
\end{thm}
\begin{proof} We check by elementary computations that if $f,g_{1},h_{1},h_{2}$ and $h$
are of the forms (1)-(5) then $(f,g_{1},h_{1},h_{2},h)$ is a
solution of (\ref{eq51}),
so left is that any solution $(f,g_{1},h_{1},h_{2},h)$ of (\ref{eq51}) fits into (1)-(5).\\
By putting $y=e$ in (\ref{eq51}) we get that
\begin{equation}\label{eq513}f=h_{1}(e)g_{1}+\dfrac{h_{2}(e)}{2}\mu_{1}+\dfrac{h_{2}(e)}{2}\mu_{2}+h(e)\chi\end{equation}
We split the discussion into the subcases $h_{1}(e)=0$ and
$h_{1}(e)\neq0$.
\par Case 1: Suppose $h_{1}(e)=0$. Then (\ref{eq513}) gives
\begin{equation}\label{eq514}f=a\mu_{1}+a\mu_{2}+b\chi\end{equation}
where $a:=\dfrac{h_{2}(e)}{2}$ and $b:=h(e)$. Hence, the set
$\{f,\mu_{1},\mu_{2},\chi\}$ is linearly dependent.
\par If $h_{1}=0$ then, according to Lemma \ref{lem53}(1), we get that
$g_{1}$ is arbitrary, $h_{2}=0$ and $f=h=b\chi$, which is solution
(1).
\par If $h_{1}\neq0$. So, by Lemma \ref{lem53}(2),
the set $\{g_{1},\mu_{1},\mu_{2},\chi\}$ is linearly dependent.
Since the set $\{\mu_{1},\mu_{2},\chi\}$ is linearly independent
there exists a triple $(\alpha,\beta,\gamma)\in\mathbb{C}^{3}$ such
that
\begin{equation}\label{eq515}g_{1}=\alpha\mu_{1}+\beta\mu_{2}+\gamma\chi.\end{equation}
By substituting (\ref{eq514}) and (\ref{eq515}) in (\ref{eq51}) we
derive by simple computations that
\begin{equation*}\begin{split}&a\mu_{1}(y)\mu_{1}(x)+a\mu_{2}(y)\mu_{2}(x)+b\chi(y)\chi(x)\\
&=[\alpha h_{1}(y)+\dfrac{1}{2}h_{2}(y)]\mu_{1}(x)+[\beta
h_{1}(y)+\dfrac{1}{2}h_{2}(y)]\mu_{2}(x)+[\gamma
h_{1}(y)+h(y)]\chi(x),\end{split}\end{equation*} for all $x,y\in G$.
So, the set $\{\mu_{1},\mu_{2},\chi\}$ being linearly independent,
we derive from the identity above that $(h_{1},h_{2},h)$ satisfies
the following identities as a system
\begin{equation}\label{eq515-1}\alpha
h_{1}+\dfrac{1}{2}h_{2}=a\mu_{1},\end{equation}
\begin{equation}\label{eq515-2}\beta
h_{1}+\dfrac{1}{2}h_{2}=a\mu_{2}\end{equation} and
\begin{equation}\label{eq515-3}\gamma h_{1}+h=b\chi.\end{equation}
By subtracting (\ref{eq515-2}) from (\ref{eq515-1}) we get that
\begin{equation}\label{eq515-4}(\alpha-\beta)h_{1}=a(\mu_{1}-\mu_{2}).\end{equation}
\par If $\alpha=\beta$ then (\ref{eq515-4}) implies that $a=0$ because
$\mu_{1}\neq\mu_{2}$. Hence, from (\ref{eq514}), (\ref{eq515}),
(\ref{eq515-1}) and (\ref{eq515-4}) we get that $$f=b\chi,\,
g_{1}=\alpha\mu_{1}+\alpha\mu_{2}+\gamma\chi,\,
h_{1}\neq0\,\,\text{is arbitrary}\,\,, h_{2}=-2\alpha
h_{1}\,\,\text{and}\,\,h=b\chi-\gamma h_{1}.$$ This is solution (2).
\par If $\alpha\neq\beta$ then we obtain, from (\ref{eq515-4}), that
\begin{equation}\label{eq516}h_{1}=\dfrac{a}{\alpha-\beta}(\mu_{1}-\mu_{2}).\end{equation}
By substituting (\ref{eq516}) in (\ref{eq515-1}) and in
(\ref{eq515-3}) we get that
\begin{equation}\label{eq517}h_{2}=-\dfrac{2\beta a}{\alpha-\beta}\mu_{1}+\dfrac{2\alpha a}{\alpha-\beta}\mu_{2}.\end{equation}
and
\begin{equation}\label{eq518}h=-\dfrac{\gamma a}{\alpha-\beta}\mu_{1}+\dfrac{\gamma a}{\alpha-\beta}\mu_{2}+b\chi.\end{equation}
Let
$$d_{1}:=2\alpha,\,d_{2}:=2\beta,\,d_{3}:=\gamma$$
$$a_{1}:=\dfrac{2a}{d_{1}-d_{2}},\, a_{2}:=-\dfrac{2a}{d_{1}-d_{2}}\,\,\text{and}\,\,a_{3}:=-\dfrac{2b}{d_{1}-d_{2}}.$$
Notice that $d_{1}\neq d_{2}$ because $\alpha-\beta\neq0$. By using
(\ref{eq514}), (\ref{eq515}), (\ref{eq516}), (\ref{eq517}) and
(\ref{eq518}) we obtain by small computations
$$f=(d_{1}-d_{2})(\dfrac{1}{2}a_{1}\mu_{1}-\dfrac{1}{2}a_{2}\mu_{2}-\dfrac{1}{2}a_{3}\chi),\,
g_{1}=\dfrac{1}{2}d_{1}\mu_{1}+\dfrac{1}{2}d_{2}\mu_{2}+d_{3}\chi,$$
$$h_{1}=a_{1}\mu_{1}+a_{2}\mu_{2},$$$$h_{2}=-a_{1}d_{2}\mu_{1}-a_{2}d_{1}\mu_{2}$$
and
$$h=-a_{1}d_{3}\mu_{1}-a_{2}d_{3}\mu_{2}-\dfrac{1}{2}a_{3}(d_{1}-d_{2})\chi,$$
which is solution (3).
\par Case 2: Suppose $h_{1}(e)\neq0$. Then from (\ref{eq513})
we get that
\begin{equation}\label{eq519}g_{1}=\dfrac{1}{h_{1}(e)}f-\dfrac{h_{2}(e)}{h_{1}(e)}g-\dfrac{h(e)}{h_{1}(e)}\chi.\end{equation}
Substituting this back into (\ref{eq51}) we obtain
$$f(xy)=[\dfrac{1}{h_{1}(e)}f(x)-\dfrac{h_{2}(e)}{h_{1}(e)}g(x)-\dfrac{h(e)}{h_{1}(e)}\chi(x)]h_{1}(y)+g(x)h_{2}(y)+\chi(x)h(y)$$
for all $x,y\in G$. So that
$$f(xy)=f(x)(\dfrac{1}{h_{1}(e)}h_{1})(y)+(h_{2}-\dfrac{h_{2}(e)}{h_{1}(e)}h_{1})(y)g(x)+\chi(x)(h-\dfrac{h(e)}{h_{1}(e)}h_{1})(y)$$
for all $x,y\in G$, i.e., the quadruple
$(f,h_{1}/h_{1}(e),h_{2}-\dfrac{h_{2}(e)}{h_{1}(e)}h_{1},h-\dfrac{h(e)}{h_{1}(e)}h_{1})$
is a solution of the functional equation (\ref{eq52}), so, according
to
Theorem \ref{thm56}, that quadruple falls into three categories:\\
(i) $f=c\mu+\varphi_{\mu,\chi},\, h_{1}/h_{1}(e)=\mu,\,
h_{2}-\dfrac{h_{2}(e)}{h_{1}(e)}h_{1}=0,\,
h-\dfrac{h(e)}{h_{1}(e)}h_{1}=\varphi_{\mu,\chi}$,\\
where $\mu$ is a character of $G$ and $c\in\mathbb{C}$ is a
constant. We define complex constants by
$b:=h_{1}(e)\in\mathbb{C}\setminus\{0\}$,
$a_{1}:=\dfrac{c}{h_{1}(e)}$, $a_{2}:=\dfrac{h_{2}(e)}{h_{1}(e)}$
and $a_{3}:=\dfrac{h(e)}{h_{1}(e)}$. So, the identities above and
(\ref{eq519}) yield, by writing $\varphi_{\mu,\chi}$ instead of
$\dfrac{1}{b}\varphi_{\mu,\chi}$, that
$$f=a_{1}b\mu+b\varphi_{\mu,\chi},\,\,
g_{1}=-\dfrac{1}{2}a_{2}\mu_{1}-\dfrac{1}{2}a_{2}\mu_{2}-a_{3}\chi+a_{1}\mu+\varphi_{\mu,\chi},$$
$$h_{1}=b\mu,\,\,h_{2}=a_{2}b\mu\,\,\text{and}\,\,
h=a_{3}b\mu+b\varphi_{\mu,\chi},$$ which is solution
(4).\\
(ii)
$$f=c\chi,\,h_{1}/h_{1}(e)\,\,\text{arbitrary},\,h_{2}-\dfrac{h_{2}(e)}{h_{1}(e)}h_{1}=0
\,\,\text{and}\,\,h-\dfrac{h(e)}{h_{1}(e)}h_{1}=c(\chi-\dfrac{1}{h_{1}(e)}h_{1}),$$
where $c\in\mathbb{C}$ is a constant.\\
Defining complex constants $b:=h_{1}(e)\in\mathbb{C}\setminus\{0\}$,
$a_{1}:=\dfrac{c}{h_{1}(e)}\in\mathbb{C}$,
$a_{2}:=\dfrac{h_{2}(e)}{h_{1}(e)}$ and
$a_{3}:=\dfrac{h(e)}{h_{1}(e)}$ we obtain, by a small computation,
from the identities above and (\ref{eq519}) that
$$f=a_{1}b\chi,\,\,g_{1}=-\dfrac{1}{2}a_{1}\mu_{1}-\dfrac{1}{2}a_{2}\mu_{2}+(a_{1}-a_{2})\chi,\,\,h_{1}\,\,\text{is arbitrary},$$
$$h_{2}=a_{2}h_{1}\,\,\text{and}\,\,h=a_{1}b\chi-(a_{1}-a_{2})h_{1},$$ which is solution (5).\\
(iii)
$$f=a\mu_{1}+b\mu_{2}+c\chi,\,\,h_{1}/h_{1}(e)=\dfrac{1}{a-b}(a\mu_{1}-b\mu_{2}),$$
$$h_{2}-\dfrac{h_{2}(e)}{h_{1}(e)}h_{1}=-\dfrac{2ab}{a-b}(\mu_{1}-\mu_{2}),\,\,h-\dfrac{h(e)}{h_{1}(e)}h_{1}=\dfrac{c}{a-b}(-a\mu_{1}+b\mu_{2}+(a-b)\chi),$$
where $a,b,c\in\mathbb{C}$ such that $a\neq b$. Using (\ref{eq519})
and the identities above we get that
\begin{equation}\label{eq520}\begin{split}&g_{1}=\dfrac{2a-h_{2}(e)}{2h_{1}(e)}\mu_{1}+\dfrac{2b-h_{2}(e)}{2h_{1}(e)}\mu_{2}+\dfrac{c-h(e)}{h_{1}(e)}\chi,\\
&h_{1}=\dfrac{ah_{1}(e)}{a-b}\mu_{1}-\dfrac{bh_{1}(e)}{a-b}\mu_{2},\\
&h_{2}=-\dfrac{a(2b-h_{2}(e))}{a-b}\mu_{1}+\dfrac{b(2a-h_{2}(e))}{a-b}\mu_{2},\\
&h=-\dfrac{a(c-h(e))}{a-b}\mu_{1}+\dfrac{b(c-h(e))}{a-b}\mu_{2}+c\chi.\end{split}\end{equation}
Defining
$$a_{1}:=\dfrac{ah_{1}(e)}{a-b},\,a_{2}:=-\dfrac{bh_{1}(e)}{a-b},\,a_{3}:=-\dfrac{ch_{1}(e)}{a-b}$$
$$d_{1}:=\dfrac{2a-h_{2}(e)}{h_{1}(e)},\,d_{2}:=\dfrac{2b-h_{2}(e)}{h_{1}(e)}\,\,\text{and}\,\,d_{3}:=\dfrac{c-h(e)}{h_{1}(e)},$$
we have $a=\frac{1}{2}a_{1}(d_{1}-d_{2})$,
$b=-\frac{1}{2}(d_{1}-d_{2})$ and
$c=-\frac{1}{2}a_{3}(d_{1}-d_{3})$. Sine $a\neq b$ we have
$d_{1}\neq d_{2}$. Moreover the identities
$f=a\mu_{1}+b\mu_{2}+c\chi$ and (\ref{eq520}) read
$$f=(d_{1}-d_{2})(\dfrac{1}{2}a_{1}\mu_{1}-\dfrac{1}{2}a_{2}\mu_{2}-\dfrac{1}{2}a_{3}\chi),\,
g_{1}=\dfrac{1}{2}d_{1}\mu_{1}+\dfrac{1}{2}d_{2}\mu_{2}+d_{3}\chi,$$
$$h_{1}=a_{1}\mu_{1}+a_{2}\mu_{2},\,h_{2}=-a_{1}d_{2}\mu_{1}-a_{2}d_{1}\mu_{2}$$
and
$$h=-a_{1}d_{3}\mu_{1}-a_{2}d_{3}\mu_{2}-\dfrac{1}{2}a_{3}(d_{1}-d_{2})\chi,$$
which is solution (3). This completes the proof of Theorem
\ref{Thm57}.
\end{proof}
\par From Case 1 of the proof of Theorem
\ref{Thm57} we derive that if $f,g_{1},h_{1},h_{2},h:G\to\mathbb{C}$
satisfy the functional equation (\ref{eq51}) and $h_{1}(e)=0$, then
$f$ and $g_{1}$ are abelian, or $f,\,g_{1},\,h_{1},\,h_{2}$ and $h$
are abelian.
\par According to \cite[Proposition 5]{Stetkaer2}, if $\mu\neq\chi$ and $G$ is an abelian group then the function $\varphi_{\mu,\chi}$, in the
formulas of solutions of the form (4) in Theorem \ref{Thm57}, take
the form $\varphi_{\mu,\chi}=c(\mu-\chi)$, where $c$ ranges over
$\mathbb{C}$. On groups that need not be abelian, the expression of
$\varphi_{\mu,\chi}$ is given by Theorem \ref{Thm53}.
\begin{exple} If we choice $b=\gamma=0$ in Theorem \ref{Thm57}(2), or $a_{3}=d_{3}=0$ in Theorem
\ref{Thm57}(3), or $a_{3}=0$ and $\varphi_{\mu,\chi}=0$ in Theorem
\ref{Thm57}(4), we get that $h=0$, then the functional equation
(\ref{eq51}) reduces to
$$f(xy)=g_{1}h_{1}(y)+g(x)h_{2}(y),\,x,y\in G$$ which was solved in \cite[Theorem 8]{Ebanks and
Stetkaer}.\\
(1) for $b=\gamma=0$ in Theorem \ref{Thm57}(2), we get, by putting
$c=2\alpha$, that
$$f=0,\,g_{1}=\alpha\mu_{1}+\alpha\mu_{2}=cg,\,h_{1}\neq0\,\,\text{is arbitrary},\,\,h_{2}=-ah_{1}\,\,\text{and}\,\,h=-bh_{1},$$
which is the solution
obtained in \cite[Theorem 8(d)]{Ebanks and Stetkaer}.\\
(2) for $a_{3}=d_{3}=0$ in Theorem \ref{Thm57}(3) we obtain
$$f=(d_{1}-d_{2})(\dfrac{1}{2}a_{1}\mu-\dfrac{1}{2}a_{2}\chi),\,
g_{1}=\dfrac{1}{2}d_{1}\mu+\dfrac{1}{2}d_{2}\chi,$$ and
$$h_{1}=a_{1}\mu+a_{2}\chi,\,h_{2}=-a_{1}d_{2}\mu-a_{2}d_{1}\chi,$$
where $a_{1},\,a_{2},\,d_{1},\,d_{2}\in\mathbb{C}$ are constants.
Furthermore for $f\neq0$ we have $d_{1}\neq d_{2}$ and
$(a_{1},a_{2})\neq(0,0)$. The solution was obtained in
\cite[Proposition 6(b), Theorem 8(b)]{Ebanks and
Stetkaer}.\\
(3) for $a_{3}=0$ and $\varphi_{\mu,\chi}=0$ in Theorem
\ref{Thm57}(4),
$$f=a_{1}b\mu,\,g_{1}=a_{1}\mu-a_{2}g,\,h_{1}=b\mu\,\,\text{and}\,\,h_{2}=a_{2}b\mu,$$
where $a_{1},\,a_{2},\,\in\mathbb{C}$ are constants. Furthermore for
$f\neq0$ we have $a_{1},\,b\in\mathbb{C}\setminus\{0\}$. The
solution was obtained in \cite[Theorem 8(a)]{Ebanks and Stetkaer}.
\end{exple}
\begin{exple} By taking $d_{1}=-d_{2}$, $a_{1}=a_{2}=\dfrac{1}{2}$
and $a_{3}=d_{3}=0$ in Theorem \ref{Thm57}(3) we get that
$$f=g_{1}=h_{2}=d_{1}\dfrac{\mu_{1}-\mu_{2}}{2}\quad\text{and}\quad
h_{1}=\dfrac{\mu_{1}+\mu_{2}}{2},\quad\text{where}\quad
d_{1}\in\mathbb{C}\setminus\{0\},$$ which is the solution of the
classic sine addition law obtained in \cite[Theorem
4.1(c)]{Stetkaer1}.
\end{exple}
\section{Solutions of Eq. (\ref{Eq5-2}) and Eq. (\ref{Eq5-1})}
In this section $\chi_{1},\cdot\cdot\cdot,\chi_{N}$ are $N$ distinct
characters of a group $G$.
\subsection{Solutions of Eq. (\ref{Eq5-2})}\quad\\\\ In Theorem \ref{thm5-1} we
solve the functional equation (\ref{Eq5-2}), i.e.,
\begin{equation*}f(xy)=f(x)h(y)+\sum_{j=1}^{N}\chi_j(x)h_j(y),\,x,y\in G,\end{equation*}
where $f,\,h,\,h_{1},\cdot\cdot\cdot,\,h_{N}:G\to\mathbb{C}$ are the
unknown functions to be determined.
\begin{thm}\label{thm5-1} The solutions $(f,\,h,\,h_{1},\cdot\cdot\cdot,\,h_{N})$ of
(\ref{Eq5-2}) are:\\
(1)
$$f=\sum_{j=1}^{N}\alpha_{j}\chi_j,\,h\,\,\text{arbitrary}\,\,\,\,\text{and}\,\,h_{j}=\alpha_{j}\chi_{j}-\alpha_{j}h$$ for all
$j=1,\cdot\cdot\cdot,N$,
where $\alpha_{1},\cdot\cdot\cdot,\alpha_{N}\in\mathbb{C}$ are constants ;\\
(2)
$$f=a\chi+\sum_{j=1}^{N}\varphi_{\chi,\chi_{j}},\,h=\chi\,\,\text{and}\,\,h_{j}=\varphi_{\chi,\chi_{j}}$$ for all $j=1,\cdot\cdot\cdot,N$,
where $\chi$ is a character of $G$ and $a\in\mathbb{C}$ is a
constant.
\end{thm}
\begin{proof} Elementary computations show that if $f,h,h_{1},\cdot\cdot\cdot,h_{N}$
are of the forms (1)-(2) then $(f,h,h_{1},\cdot\cdot\cdot,h_{N})$ is
a solution of (\ref{Eq5-2}), so left is that any solution
$(f,h,h_{1},\cdot\cdot\cdot,h_{N})$ of (\ref{Eq5-2}) fits into
(1)-(2). We consider two cases.
\par Case 1: Suppose that the set $\{f,\chi_{1},\cdot\cdot\cdot,\chi_{N}\}$ is linearly
dependent. Since $\chi_{1},\cdot\cdot\cdot,\chi_{N}$ are different
characters then, according to Proposition \ref{prop51}, the set
$\{\chi_{1},\cdot\cdot\cdot,\chi_{N}\}$ is linearly independent. So
there exist $\alpha_{1},\cdot\cdot\cdot,\alpha_{N}\in\mathbb{C}$
such that
\begin{equation*}f=\sum_{j=1}^{N}\alpha_{j}\chi_{j}.\end{equation*} Substituting this in
(\ref{Eq5-2}) we get that
\begin{equation*}\sum_{j=1}^{N}\alpha_{j}\chi_{j}(y)\chi_{j}(x)=\sum_{j=1}^{N}(\alpha_{j}h(y)+h_{j}(y))\chi_{j}(x),\end{equation*}
for all $x,y\in G$. So, by using the linear independence of
$\{\chi_{1},\cdot\cdot\cdot,\chi_{N}\}$, we obtain
\begin{equation*}\alpha_{j}h(y)+h_{j}(y)=\alpha_{j}\chi_{j}(y),\end{equation*}
for all $y\in G$ and all $j=1,\cdot\cdot\cdot,N$. Hence
\begin{equation*}h_{j}=\alpha_{j}\chi_{j}-\alpha_{j}h,\end{equation*}
for all $j=1,\cdot\cdot\cdot,N$, with $h$ arbitrary. The result
occurs in part (1).
\par Case 2: Suppose that the set $\{f,\chi_{1},\cdot\cdot\cdot,\chi_{N}\}$ is linearly
independent. Let $x,y,z\in G$ be arbitrary. We compute $f(zxy)$ by
using the associativity of the operation of $G$. We have
\begin{equation}\label{Eq5-3}f(z(xy))=f(z)h(xy)+\sum_{j=1}^{N}\chi_{j}(z)h_{j}(xy).\end{equation}
On the other hand, by using (\ref{Eq5-2}), and that
$\chi_{1},\cdot\cdot\cdot,\chi_{N}$ are multiplicative, we obtain
\begin{equation}\label{Eq5-4}\begin{split}&f((zx)y)=f(zx)h(y)+\sum_{j=1}^{N}\chi_{j}(zx)h_{j}(y)\\
&=[f(z)h(x)+\sum_{j=1}^{N}\chi_{j}(z)h_{j}(x)]h(y)+\sum_{j=1}^{N}\chi_{j}(zx)h_{j}(y)\\
&=f(z)h(x)h(y)+\sum_{j=1}^{N}\chi_{j}(z)[h_{j}(x)h(y)+\chi_{j}(x)h_{j}(y)].\end{split}\end{equation}
Since the set $\{f,\chi_{1},\cdot\cdot\cdot,\chi_{N}\}$ is linearly
independent and $x,y,z$ are arbitrary we derive from (\ref{Eq5-3})
and (\ref{Eq5-4}) that
\begin{equation}\label{Eq5-5}h(xy)=h(x)h(y)\end{equation}
and
\begin{equation}\label{Eq5-6}h_{j}(xy)=h_{j}(x)h(y)+\chi_{j}(x)h_{j}(y),\end{equation}
for all $x,y\in G$ and all $j=1,\cdot\cdot\cdot,N$. The functional
equation (\ref{Eq5-5}) says that $\chi:=h$ is a multiplicative
function of $G$. If there exists $x_{0}\in G$ such that
$\chi(x_{0})=0$ then $h=\chi=0$. So, by putting $y=e$ in
(\ref{Eq5-2}), we get that $f=\sum_{j=1}^{N}h_{j}(e)\chi_{}$, which
contradicts that the set $\{f,\chi_{1},\cdot\cdot\cdot,\chi_{N}\}$
is linearly independent. Thus $\chi$ is a character of $G$ and the
functional equations (\ref{Eq5-6}) becomes
\begin{equation}\label{Eq5-6}h_{j}(xy)=h_{j}(x)\chi(y)+\chi_{j}(x)h_{j}(y),\end{equation}
for all $x,y\in G$ and all $j=1,\cdot\cdot\cdot,N$, i.e.,
$h_{j}=\varphi_{\chi,\chi_{j}}$ for all $j=1,\cdot\cdot\cdot,N$.
\par To find $f$ we put $x=e$ in (\ref{Eq5-2}) which yields that
$$f=a\chi+\sum_{j=1}^{N}\varphi_{\chi,\chi_{j}}$$ where $a:=f(e)$. The result occurs in part
(2). This completes the proof of Theorem \ref{thm5-1}.
\end{proof}
\subsection{Solutions of Eq. (\ref{Eq5-1})}\quad\\\\
In Theorem \ref{thm5-2} we solve the functional equation
(\ref{Eq5-1}), i.e.,
\begin{equation*}f(xy)=g(x)h(y)+\sum_{j=1}^{N}\chi_j(x)h_j(y),\,x,y\in G,\end{equation*}
where $f,\,g,\,h,\,h_{1},\cdot\cdot\cdot,\,h_{N}:G\to\mathbb{C}$ are
the unknown functions to be determined.
\begin{thm}\label{thm5-2} The solutions $(f,\,g,\,h,\,h_{1},\cdot\cdot\cdot,\,h_{N})$ of
(\ref{Eq5-1}) are:\\
(1)
$$f=\sum_{j=1}^{N}a_{j}\chi_j,\,g\,\,\text{arbitrary}\,\,,h=0\,\,\text{and}\,\,h_{j}=a_{j}\chi_{j}$$
for all $j=1,\cdot\cdot\cdot,N$, where $a_{1},\cdot\cdot\cdot,a_{N}\in\mathbb{C}$ are constants;\\
(2)
$$f=\sum_{j=1}^{N}a_{j}\chi_j,\,g=\sum_{j=1}^{N}\beta_{j}\chi_j,\,h\neq0\,\,\text{arbitrary}\,\,\text{and}\,\,h_{j}=a_{j}\chi_{j}-\beta_{j}h,$$ for all
$j=1,\cdot\cdot\cdot,N$,
where $a_{1},\cdot\cdot\cdot,a_{N},\beta_{1},\cdot\cdot\cdot,\beta_{N}\in\mathbb{C}$ are constants;\\
(3) $$f=\alpha
b\chi+\sum_{j=1}^{N}b\varphi_{\chi,\chi_{j}},\,g=\alpha\chi-\sum_{j=1}^{N}a_{j}\chi_{j}+\sum_{j=1}^{N}\varphi_{\chi,\chi_{j}},\,h=b\chi\,\,\text{and}\,\,h_{j}=a_{j}b\chi+b\varphi_{\chi,\chi_{j}},$$
for all $j=1,\cdot\cdot\cdot,N$, where $\chi$ is a character of $G$,
and $\alpha,a_{1},\cdot\cdot\cdot,a_{N}\in\mathbb{C}$ and
$b\in\mathbb{C}\setminus\{0\}$ are constants.
\end{thm}
\begin{proof}  We check by elementary computations that if $f,\,g,\,h,\,h_{1},\cdot\cdot\cdot,\,h_{N}$
are of the forms (1)-(3) then
$(f,\,g,\,h,\,h_{1},\cdot\cdot\cdot,\,h_{N})$ is a solution of
(\ref{Eq5-1}),
so left is that any solution $(f,\,g,\,h,\,h_{1},\cdot\cdot\cdot,\,h_{N})$ of (\ref{Eq5-1}) fits into (1)-(3).\\
We will discuss two cases according to whether the set
$\{f,\chi_{1},\cdot\cdot\cdot,\chi_{N}\}$ is linearly independent or
not.
\par Case 1: Suppose that the set $\{f,\chi_{1},\cdot\cdot\cdot,\chi_{N}\}$ is linearly dependent.
As seen earlier the set $\{\chi_{1},\cdot\cdot\cdot,\chi_{N}\}$ is
linearly independent. Then there exist
$a_{1},\cdot\cdot\cdot,a_{N}\in\mathbb{C}$ such that
\begin{equation}\label{Eq5-25-1}f=\sum_{j=1}^{N}a_{j}\chi_{j}.\end{equation}
There are two subcases to consider.
\par Subcase 1.1: Suppose that $h=0$. Then $g$ is arbitrary, and taking
(\ref{Eq5-25-1}) into account, the functional equation (\ref{Eq5-1})
implies that
$$\sum_{j=1}^{N}a_{j}\chi_{j}(y)\chi_{j}(x)=\sum_{j=1}^{N}h_{j}(y)\chi_{j}(x),$$
for all $x,y\in G$ and all $j=1,\cdot\cdot\cdot,N$. So, the set
$\{\chi_{1},\cdot\cdot\cdot,\chi_{N}\}$ being linearly independent,
and $x$ and $y$ arbitrary, we get, according to Proposition
\ref{prop51}, that $h_{j}=a_{j}\chi_{j}$ for all
$j=1,\cdot\cdot\cdot,N$. The result occurs in part (1).
\par Subcase 1.2: Suppose that $h\neq0$. Then, we derive from Eq. (\ref{Eq5-1})
that there exist constants
$\beta_{1},\cdot\cdot\cdot,\beta_{N}\in\mathbb{C}$ such that
\begin{equation}\label{Eq5-25-2}g=\sum_{j=1}^{N}\beta_{j}\chi_{j}.\end{equation}
By substituting (\ref{Eq5-25-2}) in (\ref{Eq5-1}) we obtain by a
small computation, that
$$\sum_{j=1}^{N}\chi_{j}(x)[a_{j}\chi_{j}(y)-\beta_{j}h(y)-h_{j}(y)]=0,$$
for all $x,y\in G$. Since the set
$\{\chi_{1},\cdot\cdot\cdot,\chi_{N}\}$ is linearly independent, and
$x$ and $y$ arbitrary, we get, according to Proposition
\ref{prop51}, that
\begin{equation}\label{Eq5-25-3}h_{j}=a_{j}\chi_{j}-\beta_{j}h,\end{equation} for
all $j=1,\cdot\cdot\cdot,N$. The identities (\ref{Eq5-25-1}),
(\ref{Eq5-25-2}) and (\ref{Eq5-25-3}), with $h\neq0$ arbitrary,
constitute the result (2).
\par Case 2: Suppose that the set $\{f,\chi_{1},\cdot\cdot\cdot,\chi_{N}\}$ is linearly independent. By putting $y=e$ in
the functional equation (\ref{Eq5-1}) we get that
\begin{equation}\label{Eq5-25}f=h(e)g+\sum_{j=1}^{N}h_{j}(e)\chi_{j}.\end{equation}
Since the set $\{f,\chi_{1},\cdot\cdot\cdot,\chi_{N}\}$ is linearly
independent the identity (\ref{Eq5-25}) impose that $h(e)\neq0$,
which implies that
\begin{equation}\label{Eq5-29}g=\dfrac{1}{h(e)}f-\sum_{j=1}^{N}\dfrac{h_{j}(e)}{h(e)}\chi_{j}.\end{equation}
Substituting this back into (\ref{Eq5-1}) we get, by elementary
computation, that
$$f(xy)=f(x)\dfrac{h(y)}{h(e)}+\sum_{j=1}^{N}[h_{j}(y)-\dfrac{h_{j}(e)}{h(e)}h(y)]\chi_{j}(x)=0,$$
for all $x,y\in G$. Then the functions
$f,h/h(e),h_{1}-\dfrac{h_{1}(e)}{h(e)}h,\cdot\cdot\cdot,h_{N}-\dfrac{h_{N}(e)}{h(e)}h$
satisfy the functional equation (\ref{Eq5-2}). Hence, according to
Theorem \ref{thm5-1} and seeing that the set
$\{f,\chi_{1},\cdot\cdot\cdot,\chi_{N}\}$ is linearly independent,
we derive that:
$$f=a\chi+\sum_{j=1}^{N}\varphi_{\chi,\chi_{j}},\,h/h(e)=\chi\,\,\text{and}\,\,h_{j}-\dfrac{h_{j}(e)}{h(e)}h=\varphi_{\chi,\chi_{j}}$$ for all $j=1,\cdot\cdot\cdot,N$,
where $\chi$ is a character of $G$ and $a\in\mathbb{C}$ is a
constant. Using this and the identity (\ref{Eq5-29}) we get that
$$f=a\chi+\sum_{j=1}^{N}\varphi_{\chi,\chi_{j}},$$
$$g=\dfrac{a}{h(e)}\chi+\sum_{j=1}^{N}\dfrac{1}{h(e)}\varphi_{\chi,\chi_{j}}-\sum_{j=1}^{N}\dfrac{h_{j}(e)}{h(e)}\chi_{j},$$
$$h=h(e)\chi$$
and
$$h_{j}=h_{j}(e)\chi+\varphi_{\chi,\chi_{j}},$$
for all $j=1,\cdot\cdot\cdot,N$. Defining complex constants
$$b:=h(e)\neq0,\,\alpha:=\dfrac{a}{h(e)}\,\quad\text{and}\quad
a_{j}:=\dfrac{h_{j}(e)}{h(e)},$$  and written
$\varphi_{\chi,\chi_{j}}$ instead of
$\dfrac{1}{b}\varphi_{\chi,\chi_{j}}$ for all
$j=1,\cdot\cdot\cdot,N$, the formulas above read
$$f=\alpha
b\chi+\sum_{j=1}^{N}b\varphi_{\chi,\chi_{j}},\,g=\alpha\chi-\sum_{j=1}^{N}a_{j}\chi_{j}+\sum_{j=1}^{N}\varphi_{\chi,\chi_{j}},\,h=b\chi\,\,\text{and}\,\,h_{j}=a_{j}b\chi+b\varphi_{\chi,\chi_{j}},$$
for all $j=1,\cdot\cdot\cdot,N$. The result occurs in part (3). This
completes the proof.
\end{proof}
\par We close this section with relating our results to two
functional equations in the literature. The proof of Theorem
\ref{thm5-2} reveals that if the set
$\{f,\chi_{1},\cdot\cdot\cdot,\chi_{N}\}$ is linearly independent,
then the solutions $(f,\,g,\,h,\,h_{1},\cdot\cdot\cdot,\,h_{N})$ of
the functional equation (\ref{Eq5-1}) are of the form (3) in Theorem
\ref{thm5-2}. This is an extension of \cite[Theorem 14]{Stetkaer2}
as example \ref{example1} says.\\
Example \ref{example2} relate Theorem \ref{thm5-2}(3) to the simple
generalization of Prexider's functional equation
$$f(xy)=g_{1}(x)\chi(y)+\chi(x)h_{1}(y),\,\,x,y\in G,$$ where $\chi$
is a character of $G$ (see \cite[Proposition 15]{Stetkaer2}).
\begin{exple}\label{example1}
If we choice $a_{j}=0$ and $\varphi_{\chi,\chi_{j}}=0$ for all
$j=2,\cdot\cdot\cdot,N$ in part (3) of Theorem \ref{thm5-2} we get
that $h_{j}=0$ for all $j=2,\cdot\cdot\cdot,N$. With the notations
$c_{1}:=\alpha b,\,c_{2}:=a_{1}b-c_{1}$ and
$f_{\chi,\chi_{1}}:=c_{1}(\chi-\chi_{1})+b\varphi_{\chi,\chi_{1}}$
we get that
\begin{equation*}f=c_{1}\mu+f_{\chi,\chi_{1}},\,g=(f_{\chi,\chi_{1}}-c_{2}\mu)/b,\,h=b\chi
\quad\text{and}\quad
h_{1}=c_{1}\chi_{1}+c_{2}\chi+f_{\chi,\chi_{1}},\end{equation*}
where $b\in\mathbb{C}\setminus\{0\}$ and $c_{1},c_{2}\in\mathbb{C}$
are constants, and $f_{\chi,\chi_{1}}:G\to\mathbb{C}$ satisfies the
functional equation
$$f_{\chi,\chi_{1}}(xy)=f_{\chi,\chi_{1}}(x)\chi(y)+\chi_{1}(x)f_{\chi,\chi_{1}}(y),\,\,x,y\in G,$$
which is the solution obtained in \cite[Theorem 14]{Stetkaer2}.
\end{exple}
\begin{exple}\label{example2}
If we take $b=1$, $\chi=\chi_{1}$, $a_{j}=0$ and
$\varphi_{\chi,\chi_{j}}=0$ for all $j=2,\cdot\cdot\cdot,N$ in part
(3) of Theorem \ref{thm5-2}, we get that $h_{j}=0$ for all
$j=2,\cdot\cdot\cdot,N$ and
$\varphi_{\chi,\chi_{1}}=\varphi_{\chi,\chi}=\chi A$ where $A$ is an
additive function on $G$. With the notations
$c_{1}:=\alpha-a_{1}\,\text{and}\,\, c_{2}:=a_{1}$  we obtain
\begin{equation*}f=(c_{1}+c_{2}+A)\chi,\,g=(c_{1}+A)\chi\,
\quad\text{and}\quad h_{1}=(c_{2}+A)\chi,\end{equation*} where
$c_{1},c_{2}\in\mathbb{C}$ are constants, which is the solution
obtained in \cite[Proposition 15]{Stetkaer2}.
\end{exple}
\section{Solutions of Eq. (\ref{equ52}) and Eq. (\ref{equ51})}
Throughout this section $G$ denotes a monoid with identity element
$e$, $\mu,\chi:G\to\mathbb{C}$ different nonzero multiplicative
functions, $A:G\to\mathbb{C}$ an additive function such that $\chi
A\neq0$ and $g:=\dfrac{\mu+\chi}{2}$.
\subsection{Solutions of Eq. (\ref{equ52})}\quad\\\\Proposition
\ref{Prop56} solves the functional equation (\ref{equ52}), i.e.,
\begin{equation*}f(xy)=f(x)h_{1}(y)+g(x)h_{2}(y)+\chi(x)A(x)h(y),\,x,y\in G,\end{equation*}under
the assumption that the set $\{f,\mu,\chi,\chi\,A\}$ is linearly
independent.
\begin{prop}\label{Prop56} The solutions $(f,h_{1},h_{2},h)$ of
(\ref{equ52}) such that the set $\{f,\mu,\chi,\chi\,A\}$ is linearly
independent are:\\
$$f=c m,\,h_{1}=m,\,h_{2}=h=0,$$ where $m$ is a non zero
multiplicative function and $c\in\mathbb{C}\setminus\{0\}$ is a
constant such that $m\neq\mu$ and $m\neq\chi$.
\end{prop}
\begin{proof} It is easy to check that the formulas in Proposition
\ref{Prop56} define solutions of Eq. (\ref{equ52}), so left is that
any
solution $(f,h_{1},h_{2},h)$ is of that form.\\
Let $x,y,z\in G$ be arbitrary. First we compute $f(xyz)$ as
$f(x(yz))$ and then as $f((xy)z)$. Using Eq. (\ref{equ52}) and that
$g=\dfrac{\mu+\chi}{2}$ we obtain
\begin{equation}\label{equ54}f(x(yz))=f(x)h_{1}(yz)+\dfrac{1}{2}\mu(x)h_{2}(yz)+\dfrac{1}{2}\chi(x)h_{2}(yz)+\chi(x)A(x)h(yz).\end{equation}
On the other hand, by making use that $\mu$ and $\chi$ are
multiplicative and that $A$ is additive, we get from Eq.
(\ref{equ52}) that
\begin{equation*}\begin{split}f((xy)z)=f(xy)h_{1}(z)+g(xy)h_{2}(z)+\chi(xy)A(xy)h(z)\quad\quad\quad\quad\quad\quad\quad\quad\quad\quad\\
=[f(x)h_{1}(y)+\dfrac{1}{2}\mu(x)h_{2}(y)+\dfrac{1}{2}\chi(x)h_{2}(y)+\chi(x)A(x)h(y)]h_{1}(z)+\dfrac{1}{2}\mu(x)\mu(y)h_{2}(z)\\
+\dfrac{1}{2}\chi(x)\chi(y)h_{2}(z)+\chi(x)\chi(y)A(y)h(z)+\chi(x)A(x)\chi(y)h(z).\quad\quad\quad\quad\quad\quad\quad\quad\end{split}\end{equation*}
So that
\begin{equation}\label{equ55}\begin{split}f((xy)z)&=f(x)h_{1}(y)h_{1}(z)+\dfrac{1}{2}\mu(x)[h_{2}(y)h_{1}(z)+\mu(y)h_{2}(z)]\\
&+\dfrac{1}{2}\chi(x)[h_{2}(y)h_{1}(z)+\chi(y)h_{2}(z)+2\chi(y)A(y)h(z)]\\&+\chi(x)A(x)[h(y)h_{1}(z)+\chi(y)h(z)].\end{split}\end{equation}
As the set $\{f,\mu,\chi,\chi\,A\}$ is linearly independent and
$x,y,z$ are arbitrary, we derive from (\ref{equ54}) and
(\ref{equ55}) that $(h_{1},h_{2},h)$ is a solution of the following
functional equations as a system
\begin{equation}\label{equ56}h_{1}(xy)=h_{1}(x)h_{1}(y),\end{equation}
\begin{equation}\label{equ57}h_{2}(xy)=h_{2}(x)h_{1}(y)+\mu(x)h_{2}(y)\end{equation}
and
\begin{equation}\label{equ58}h_{2}(xy)=h_{2}(x)h_{1}(y)+\chi(x)h_{2}(y)+2\chi(x)A(x)h(y),\end{equation}
for all $x,y\in G$. The functional equation (\ref{equ56}) implies
that $h_{1}=m$ where $m:G\rightarrow\mathbb{C}$ is a multiplicative
function. From (\ref{equ57}) and (\ref{equ58}) we get that
$$[\mu(x)-\chi(x)]h_{2}(y)=2\chi(x)A(x)h(y)$$ for all $x,y\in G$,
from which we deduce that $h(y)=h_{2}(y)=0$ for all $y\in G$, and
then $h_{2}=h=0$. Hence, the functional equation (\ref{equ52})
reduces to
$$f(xy)=f(x)m(y)$$ for all $x,y\in G$, which yields, by putting
$x=e$, that $$f=c m\,\,\text{where}\,\,c:=f(e).$$ If $c=0$ or
$m\in\{0,\mu,\chi\}$, then the set $\{f,\mu,\chi,\chi\,A\}$ is
linearly dependent, contradicting the hypothesis. Thus
$c\in\mathbb{C}\setminus\{0\}$, $m\neq0$, $m\neq\mu$ and
$m\neq\chi$.
\end{proof}
\begin{thm}\label{thm56} The solutions $(f,h_{1},h_{2},h)$ of
(\ref{equ52}) are:\\
(1) $f=cm,\,h_{1}=m,\,h_{2}=h=0$, where $m$ is a non zero
multiplicative function and $c\in\mathbb{C}\setminus\{0\}$ is a
constant such that $m\neq\mu$ and $m\neq\chi$;\\
(2) $f=0,\,h_{1}\,\,\text{is arbitrary}\,\,\text{and}\,\,
h_{2}=h=0$;\\
(3)$$f=a\mu+b\chi+c\chi\,A,\,\,h_{1}=\dfrac{1}{a-b}(a\mu-b\chi-c\chi\,A),$$
$$h_{2}=\dfrac{2a}{a-b}(-b\mu+b\chi+c\chi\,A),\,\,h=\dfrac{c}{a-b}(-a\mu+a\chi+c\chi\,A),$$
where $a,b,c\in\mathbb{C}$ are constants such that
$(a,b,c)\neq(0,0,0)$ and $a\neq b$.
\end{thm}
\begin{proof} (1) If the set $\{f,\mu,\chi,\chi\,A\}$ is linearly independent we get, by Proposition \ref{Prop56},
the part (1) of Theorem \ref{thm56}.
\par In the remainder of the proof we assume that $\{f,\mu,\chi,\chi\,A\}$ is linearly dependent.\\(2) If $f=0$ then $h_{1}$ is arbitrary and
the functional equation (\ref{equ52}) implies that
$$\dfrac{h_{2}(y)}{2}\mu(x)+\dfrac{h_{2}(y)}{2}\chi(x)+h(y)(\chi\,A)(x)=0$$
for all $x,y\in G$. So $h_{2}=h=0$. In
what follows we assume that $f\neq0$.\\
(3) By using Lemma \ref{lem52} we get that the set
$\{\mu,\chi,\chi\,A\}$ is linearly independent. So, seeing that
$f\neq0$ there exists a triple
$(a,b,c)\in\mathbb{C}^{3}\setminus\{(0,0,0)\}$ such that
$$f=a\mu+b\chi+c\chi\,A.$$ Substituting this back into (\ref{equ52})
we get by a small computation that
\begin{equation*}\begin{split}&(a\mu(y))\mu(x)+[b\chi(y)+c\chi(y)A(y)]\chi(x)+(c\chi(y))(\chi\,A)(x)\\
&=[ah_{1}(y)+\dfrac{1}{2}h_{2}(y)]\mu(x)+[ah_{1}(y)+\dfrac{1}{2}h_{2}(y)]\chi(x)+[ch_{1}(y)+h(y)](\chi\,A)(x),\end{split}\end{equation*}
for all $x,y\in G$. We derive that $(h_{1},h_{2},h)$ is a solution
of the following identities as a system
\begin{equation}\label{equ59}ah_{1}+\dfrac{1}{2}h_{2}=a\mu,\end{equation}
\begin{equation}\label{equ510}bh_{1}+\dfrac{1}{2}h_{2}=b\chi+c\chi\,A,\end{equation}
and
\begin{equation}\label{equ511}ch_{1}+h=c\chi.\end{equation}
When we subtract (\ref{equ510}) from (\ref{equ59}) we obtain
\begin{equation}\label{equ512}(a-b)h_{1}=a\mu-b\chi-c\chi\,A.\end{equation} We have $a\neq b$. Indeed, if
$a=b$ the identity (\ref{equ512}) reduces to
$$a\mu-b\chi-c\chi\,A=0.$$ Hence $a=b=c=0$, which contradicts the
assumption on $f$. Hence, the identity (\ref{equ512}) implies that
$$h_{1}=\dfrac{1}{a-b}(a\mu-b\chi-c\chi\,A).$$
Substituting this back into (\ref{equ59}) and (\ref{equ511}) we
derive that
$$h_{2}=2a(\mu-h_{1})=2a[\mu-\dfrac{1}{a-b}(a\mu-b\chi-c\chi\,A)]=\dfrac{2a}{a-b}(-b\mu+b\chi+c\chi\,A)$$
and
$$h=c(\chi-h_{1})=c[\chi-\dfrac{1}{a-b}(a\mu-b\chi-c\chi\,A)]=\dfrac{c}{a-b}(-a\mu+a\chi+c\chi\,A).$$
\par Conversely, simple computations show that the formulas in
(1)-(3) define solutions of (\ref{equ52}).
\end{proof}
\subsection{Solutions of Eq. (\ref{equ51})}\quad\\\\
In Theorem \ref{thm57} we solve the functional equation
(\ref{equ51}), i.e.,
\begin{equation*}f(xy)=g_{1}(x)h_{1}(y)+g(x)h_{2}(y)+\chi(x)A(x)h(y),\quad x,y\in
G,\end{equation*} where
$f,g_{1},h_{1},h_{2},h:G\rightarrow\mathbb{C}$ are the unknown
functions to be determined.
\begin{thm}\label{thm57} The solutions $(f,g_{1},h_{1},h_{2},h)$ of
(\ref{equ51}) can be listed as follows:\\
(1) $f=0$, $g_{1}$ is arbitrary and $h_{1}=h_{2}=h=0$;\\
(2) $f=0$, $g_{1}=a\dfrac{\mu+\chi}{2}+b\,\chi\,A$, $h_{1}\neq0$ is
arbitrary, $h_{2}=-ah_{1}$ and $h=-bh_{1}$, where $a,b\in\mathbb{C}$ are constants.\\
(3)
$$f=(d_{1}-d_{2})(\dfrac{1}{2}a_{1}\mu-\dfrac{1}{2}a_{2}\chi-\dfrac{1}{2}a_{3}\chi\,A),\,
g_{1}=\dfrac{1}{2}d_{1}\mu+\dfrac{1}{2}d_{2}\chi+d_{3}\chi\,A,$$
$$h_{1}=a_{1}\mu+a_{2}\chi+a_{3}\chi\,A,\,h_{2}=-a_{1}d_{2}\mu-a_{2}d_{1}\chi-a_{3}d_{1}\chi\,A$$
and
$$h=-a_{1}d_{3}\mu+(-a_{2}d_{3}-\dfrac{1}{2}a_{3}d_{1}+\dfrac{1}{2}a_{3}d_{2})\chi-a_{3}d_{3}\chi\,A,$$
where $a_{1},a_{2},a_{3},d_{1},d_{2},d_{3}\in\mathbb{C}$ are
constants
such that $d_{1}\neq d_{2}$ and $(a_{1},a_{2},a_{3})\neq(0,0,0)$.\\
(4) $f=a_{1}\alpha m$, $g_{1}=\alpha
m-a_{2}\dfrac{\mu+\chi}{2}-a_{3}\chi\,A$, $h_{1}=a_{1}m$,
$h_{2}=a_{1}a_{2}m$ and $h=a_{1}a_{3}m$, where
$\alpha,a_{1}\in\mathbb{C}\setminus\{0\}$ and
$a_{2},a_{3}\in\mathbb{C}$ are constants, and
$m:G\rightarrow\mathbb{C}$ is a non zero multiplicative function.
\end{thm}
\begin{proof} We check by elementary computations that if $f,g_{1},h_{1},h_{2}$ and $h$
are of the forms (1)-(4) then $(f,g_{1},h_{1},h_{2},h)$ is a
solution of (\ref{equ51}),
so left is that any solution $(f,g_{1},h_{1},h_{2},h)$ of (\ref{equ51}) fits into (1)-(4).\\
There are two cases to consider.
\par Case 1. Suppose $f=0$. By Proposition \ref{prop54}, we
obtain solutions (1) and (2).
\par Case 2. Suppose $f\neq0$. Taking $y=e$ in (\ref{equ51}) we get that
\begin{equation}\label{equ513}f=h_{1}(e)g_{1}+\dfrac{h_{2}(e)}{2}\mu+\dfrac{h_{2}(e)}{2}\chi+h(e)\chi\,A\end{equation}
We split the discussion into the subcases $h_{1}(e)=0$ and
$h_{1}(e)\neq0$.
\par Subcases 2.1. Suppose $h_{1}(e)=0$. Then (\ref{equ513}) gives
\begin{equation}\label{equ514}f=a\mu+a\chi+b\chi\,A\end{equation}
where $a:=\dfrac{h_{2}(e)}{2}$, $b:=h(e)$ and $(a,b)\neq(0,0)$.
Hence, the set $\{f,\mu,\chi,\chi\,A\}$ is linearly dependent. As
$f\neq0$ we get, according to Lemma \ref{lem53}(1), that
$h_{1}\neq0$. So, by Lemma \ref{lem53}(2), the set
$\{g_{1},\mu,\chi,\chi\,A\}$ is linearly dependent. Since the set
$\{\mu,\chi,\chi\,A\}$ is linearly independent there exists a triple
$(\alpha,\beta,\gamma)\in\mathbb{C}^{3}$ such that
\begin{equation}\label{equ515}g_{1}=\alpha\mu+\beta\chi+\gamma\chi\,A.\end{equation}
By substituting (\ref{equ514}) and (\ref{equ515}) in (\ref{equ51})
we derive by simple computations that
\begin{equation*}\begin{split}&(a\mu(y))\mu(x)+[a\chi(y)+b\chi(y)A(y)]\chi(x)+(b\chi(y))(\chi\,A)(x)\\
&=[\alpha h_{1}(y)+\dfrac{1}{2}h_{2}(y)]\mu(x)+[\beta
h_{1}(y)+\dfrac{1}{2}h_{2}(y)]\chi(x)+[\gamma
h_{1}(y)+h(y)](\chi\,A)(x),\end{split}\end{equation*} for all
$x,y\in G$. So, the set $\{\mu,\chi,\chi\,A\}$ being linearly
independent, we derive from the identity above that
$(h_{1},h_{2},h)$ satisfies the following identities as a system
$$\alpha h_{1}+\dfrac{1}{2}h_{2}=a\mu,$$ $$\beta
h_{1}+\dfrac{1}{2}h_{2}=a\chi+b\chi\,A$$ and $$\gamma
h_{1}+h=b\chi.$$ The first and the second identities imply that
$(\alpha-\beta)h_{1}=a\mu-a\chi-b\chi\,A$. Since $(a,b)\neq(0,0)$
and $h_{1}\neq0$ we have $\alpha-\beta\neq0$, so that
\begin{equation}\label{equ516}h_{1}=\dfrac{a}{\alpha-\beta}\mu-\dfrac{a}{\alpha-\beta}\chi-\dfrac{b}{\alpha-\beta}\chi\,A.\end{equation}
Substituting this in the first identity we get that
$$h_{2}=2a\mu-\dfrac{2\alpha a}{\alpha-\beta}\mu+\dfrac{2\alpha a}{\alpha-\beta}\chi+\dfrac{2\alpha
b}{\alpha-\beta}\chi\,A,$$ hence,
\begin{equation}\label{equ517}h_{2}=-\dfrac{2\beta a}{\alpha-\beta}\mu+\dfrac{2\alpha a}{\alpha-\beta}\chi+\dfrac{2\alpha
b}{\alpha-\beta}\chi\,A.\end{equation} The third identity implies
that
$$h=b\chi-\dfrac{a}{\alpha-\beta}\mu-\dfrac{a}{\alpha-\beta}\chi-\dfrac{b}{\alpha-\beta}\chi\,A$$
so that
\begin{equation}\label{equ518}h=-\dfrac{\gamma a}{\alpha-\beta}\mu+[b+\dfrac{\gamma a}{\alpha-\beta}]\chi+\dfrac{\gamma b}{\alpha-\beta}\chi\,A.\end{equation}
Let
$$d_{1}:=2\alpha,\,d_{2}:=2\beta,\,d_{3}:=\gamma$$
$$a_{1}:=\dfrac{2a}{d_{1}-d_{2}},\, a_{2}:=-\dfrac{2a}{d_{1}-d_{2}}\,\,\text{and}\,\,a_{3}:=-\dfrac{2b}{d_{1}-d_{2}}.$$
Notice that $d_{1}\neq d_{2}$ and $(a_{1},a_{2},a_{3})\neq(0,0,0)$
because $\alpha-\beta\neq0$ and $(a,b)\neq(0,0)$. From
(\ref{equ514}), (\ref{equ515}), (\ref{equ516}) and (\ref{equ517}) we
get that
$$f=(d_{1}-d_{2})(\dfrac{1}{2}a_{1}\mu-\dfrac{1}{2}a_{2}\chi-\dfrac{1}{2}a_{3}\chi\,A),\,
g_{1}=\dfrac{1}{2}d_{1}\mu+\dfrac{1}{2}d_{2}\chi+d_{3}\chi\,A,$$
$$h_{1}=a_{1}\mu+a_{2}\chi+a_{3}\chi\,A$$ and $$h_{2}=-a_{1}d_{2}\mu-a_{2}d_{1}\chi-a_{3}d_{1}\chi\,A.$$
From (\ref{equ518}) we obtain
\begin{equation*}\begin{split}h&=-a_{1}d_{3}\mu+[-\dfrac{1}{2}(d_{1}-d_{2})a_{3}+a_{1}d_{3}]\chi-a_{3}d_{3}\chi\,A\\
&=-a_{1}d_{3}\mu+(-a_{2}d_{3}-\dfrac{1}{2}a_{3}d_{1}+\dfrac{1}{2}a_{3}d_{2})\chi-a_{3}d_{3}\chi\,A,\end{split}\end{equation*}
which is solution (3).
\par Subcases 2.2. Suppose $h_{1}(e)\neq0$. Then from (\ref{equ513})
we get that
\begin{equation}\label{equ519}g_{1}=\dfrac{1}{h_{1}(e)}f-\dfrac{h_{2}(e)}{h_{1}(e)}g-\dfrac{h(e)}{h_{1}(e)}\chi\,A.\end{equation}
Substituting this back into (\ref{equ51}) we obtain
$$f(xy)=[\dfrac{1}{h_{1}(e)}f(x)-\dfrac{h_{2}(e)}{h_{1}(e)}g(x)-\dfrac{h(e)}{h_{1}(e)}\chi(x)A(x)]h_{1}(y)+g(x)h_{2}(y)+\chi(x)A(x)h(y)$$
for all $x,y\in G$. So that
$$f(xy)=f(x)(\dfrac{1}{h_{1}(e)}h_{1})(y)+(h_{2}-\dfrac{h_{2}(e)}{h_{1}(e)}h_{1})(y)g(x)+\chi(x)A(x)(h-\dfrac{h(e)}{h_{1}(e)}h_{1})(y)$$
for all $x,y\in G$, i.e., the quadruple
$(f,h_{1}/h_{1}(e),h_{2}-\dfrac{h_{2}(e)}{h_{1}(e)}h_{1},h-\dfrac{h(e)}{h_{1}(e)}h_{1})$
is a solution of the functional equation (\ref{equ52}), so,
according to
Theorem \ref{thm56}, it falls into two categories:\\
(i) $f=c m,\, h_{1}/h_{1}(e)=m,\,
h_{2}-\dfrac{h_{2}(e)}{h_{1}(e)}h_{1}=0,\,
h-\dfrac{h(e)}{h_{1}(e)}h_{1}=0$,\\
where $m$ is a non zero multiplicative function and
$c\in\mathbb{C}\setminus\{0\}$ is a constant such that $m\neq\mu$
and $m\neq\chi$. We define complex constants by
$\alpha:=\dfrac{c}{h_{1}(e)}\in\mathbb{C}\setminus\{0\}$,
$a_{1}:=h_{1}(e)\in\mathbb{C}\setminus\{0\}$,
$a_{2}:=\dfrac{h_{2}(e)}{h_{1}(e)}$ and
$a_{3}:=\dfrac{h(e)}{h_{1}(e)}$. So, the identities above and
(\ref{equ519}) yield that\\$f=a_{1}\alpha m$, $g_{1}=\alpha
m-a_{2}\dfrac{\mu+\chi}{2}-a_{3}\chi\,A$, $h_{1}=a_{1}m$,
$h_{2}=a_{1}a_{2}m$ and $h=a_{1}a_{3}m$, which is solution
(4).\\
(ii)
$$f=a\mu+b\chi+c\chi\,A,\,\,h_{1}/h_{1}(e)=\dfrac{1}{a-b}(a\mu-b\chi-c\chi\,A),$$
$$h_{2}-\dfrac{h_{2}(e)}{h_{1}(e)}h_{1}=\dfrac{2a}{a-b}(-b\mu+b\chi+c\chi\,A),\,\,h-\dfrac{h(e)}{h_{1}(e)}h_{1}=\dfrac{c}{a-b}(-a\mu+a\chi+c\chi\,A),$$
where $a,b,c\in\mathbb{C}$ such that $(a,b,c)\neq(0,0,0)$ and $a\neq
b$. Taking (\ref{equ519}) into account and using the identities
above we get, by simple computations, that
\begin{equation}\label{equ520}\begin{split}&f=a\mu+b\chi+c\chi\,A,\\
&g_{1}=\dfrac{2a-h_{2}(e)}{2h_{1}(e)}\mu+\dfrac{2b-h_{2}(e)}{2h_{1}(e)}\chi+\dfrac{c-h(e)}{h_{1}(e)}\chi\,A,\\
&h_{1}=\dfrac{ah_{1}(e)}{a-b}\mu-\dfrac{bh_{1}(e)}{a-b}\chi-\dfrac{ch_{1}(e)}{a-b}\chi\,A,\\
&h_{2}=-\dfrac{a(2b-h_{2}(e))}{a-b}\mu+\dfrac{b(2a-h_{2}(e))}{a-b}\chi+\dfrac{c(2a-h_{2}(e))}{a-b}\chi\,A,\\
&h=-\dfrac{a(c-h(e))}{a-b}\mu+(c+\dfrac{b(c-h(e))}{a-b})\chi+\dfrac{c(c-h(e))}{a-b}\chi\,A.\end{split}\end{equation}
Defining
$$a_{1}:=\dfrac{ah_{1}(e)}{a-b},\,a_{2}:=-\dfrac{bh_{1}(e)}{a-b},\,a_{3}:=-\dfrac{ch_{1}(e)}{a-b}$$
$$d_{1}:=\dfrac{2a-h_{2}(e)}{h_{1}(e)},\,d_{2}:=\dfrac{2b-h_{2}(e)}{h_{1}(e)}\,\,\text{and}\,\,d_{3}:=\dfrac{c-h(e)}{h_{1}(e)},$$
we have $d_{1}\neq d_{2}$ because $a\neq b$,
$(a_{1},a_{2},a_{3})\neq(0,0,0)$ because $(a,b,c)\neq(0,0,0)$,
$c=-\dfrac{1}{2}a_{3}d_{1}+\dfrac{1}{2}a_{3}d_{2}$, then
(\ref{equ520}) reads
$$f=(d_{1}-d_{2})(\dfrac{1}{2}a_{1}\mu-\dfrac{1}{2}a_{2}\chi-\dfrac{1}{2}a_{3}\chi\,A),\,
g_{1}=\dfrac{1}{2}d_{1}\mu+\dfrac{1}{2}d_{2}\chi+d_{3}\chi\,A,$$
$$h_{1}=a_{1}\mu+a_{2}\chi+a_{3}\chi\,A,\,h_{2}=-a_{1}d_{2}\mu-a_{2}d_{1}\chi-a_{3}d_{1}\chi\,A$$
and
$$h=-a_{1}d_{3}\mu+(-a_{2}d_{3}-\dfrac{1}{2}a_{3}d_{1}+\dfrac{1}{2}a_{3}d_{2})\chi-a_{3}d_{3}\chi\,A,$$
which is solution (3).
\end{proof}
\begin{rem} If $f,g_{1},h_{1},h_{2},h:G\rightarrow\mathbb{C}$
satisfy the functional equation (\ref{equ51}), the formulas in parts
(3) and (4) of Theorem \ref{thm57} reveal that if $f\neq0$ the
functions $f,\,g_{1},\,h_{1},\,h_{2}$ and $h$ are abelian.
\end{rem}
\begin{exple} If we choice $a_{3}=d_{3}=0$ in Theorem \ref{thm57}(3)
or $a_{3}=0$ in Theorem \ref{thm57}(4) we get that $h=0$, then the
functional equation (\ref{equ51}) reduces to
$$f(xy)=g_{1}h_{1}(y)+g(x)h_{2}(y),\,x,y\in G$$ which was solved in \cite[Theorem 8]{Ebanks and
Stetkaer}, and we obtain:\\
(a) for $a_{3}=d_{3}=0$ in Theorem \ref{thm57}(3),
$$f=(d_{1}-d_{2})(\dfrac{1}{2}a_{1}\mu-\dfrac{1}{2}a_{2}\chi),\,
g_{1}=\dfrac{1}{2}d_{1}\mu+\dfrac{1}{2}d_{2}\chi,$$ and
$$h_{1}=a_{1}\mu+a_{2}\chi,\,h_{2}=-a_{1}d_{2}\mu-a_{2}d_{1}\chi,$$
which is the solution obtained in \cite[Theorem 8(b)]{Ebanks and
Stetkaer}.\\
(b) for $a_{3}=0$ in Theorem \ref{thm57}(4),
  $$f=a_{1}\alpha m,\,g_{1}=\alpha
  m-a_{2}\dfrac{\mu+\chi}{2},\,h_{1}=a_{1}m,\quad\text{and}\quad h_{2}=a_{1}a_{2}m,$$
where $\alpha,a_{1},a_{2}\in\mathbb{C}$ are constants and
$m:G\rightarrow\mathbb{C}$ is a non zero multiplicative function
such that $\alpha\neq0$ and $a_{1}\neq0$, which is the solution
obtained in \cite[Theorem 8(a)]{Ebanks and Stetkaer}.
\end{exple}
\begin{exple} By taking $d_{1}=-d_{2}$, $a_{1}=a_{2}=\dfrac{1}{2}$
and $a_{3}=d_{3}=0$ in Theorem \ref{thm57}(3) we get that
$$f=g_{1}=h_{2}=d_{1}\dfrac{\mu-\chi}{2}\quad\text{and}\quad
h_{2}=\dfrac{\mu+\chi}{2},\quad\text{where}\quad
d_{1}\in\mathbb{C}\setminus\{0\},$$ which is the solution of the
classic sine addition law in case $\mu\neq\chi$ obtained in
\cite[Theorem 4.1(c)]{Stetkaer1}.
\end{exple}
\section{Solutions of Eq. (\ref{Eq.5-2}) and Eq. (\ref{Eq.5-1})}
Throughout this section $G$ denotes a group with identity element
$e$, $\mu,\chi:G\to\mathbb{C}^{*}$ different characters and
$A:G\to\mathbb{C}$ a nonzero additive function. We deal with the
functional equations (\ref{Eq.5-2}) and (\ref{Eq.5-1})
\subsection{Solutions of Eq. (\ref{Eq.5-2})}\quad\\\\
In Theorem \ref{thm7-1} we solve the functional equation
(\ref{Eq.5-2}), i.e.,
\begin{equation*}f(xy)=f(x)h_{1}(y)+\mu(x)h_{2}(y)+\chi(x)A(x)h(y),\,x,y\in G,\end{equation*}
where $f,h_{1},h_{2},h$ are the unknown functions to be determined.
\begin{thm}\label{thm7-1} The solutions $(f,h_{1},h_{2},h)$ of
(\ref{Eq.5-2}) are:\\
(1) $f=a\mu$, $h_{1}$ arbitrary, $h_{2}=a\mu-ah_{1}$ and $h=0$, where $a$ ranges over $\mathbb{C}$;\\
(2) $f=a\mu_{1}+\varphi_{\mu_{1},\mu}$, $h_{1}=\mu_{1}$,
$h_{2}=\varphi_{\mu_{1},\mu}$ and $h=0$, where $\mu_{1}$ is a
character of $G$ and $a$ ranges over
$\mathbb{C}$;\\
(3) $f=a\mu+b\chi+b\alpha\chi\,A$, $h_{1}=\chi+\alpha\chi\,A$,
$h_{2}=a\mu-a\chi-a\alpha\chi\,A$ and $h=-\alpha^{2} b\chi\,A$,
where $a$ and $\alpha$ range over $\mathbb{C}$, and $b$ over
$\mathbb{C}\setminus\{0\}$.
\end{thm}
\begin{proof} Elementary computations show that if $f,h_{1},h_{2}$ and $h$
are of the forms (1)-(3) then $(f,h_{1},h_{2},h)$ is a solution of
(\ref{Eq.5-2}), so left is that any solution $(f,h_{1},h_{2},h)$ of
(\ref{Eq.5-2}) fits into (1)-(3).\\If $f=0$ then $h_{1}$ is
arbitrary and the functional equation (\ref{Eq.5-2}) becomes
$\mu(x)h_{2}(y)+\chi(x)A(x)h(y)=0$ for all $x,y\in G$. Since the set
$\{\mu,\chi\,A\}$ is linearly independent we get that $h_{2}=h=0$,
which occurs in (1) for $a=0$. In the rest of the proof we assume
that $f\neq0$. We consider two cases.
\par Case 1. Suppose that the set $\{f,\mu,\chi,\chi\,A\}$ is linearly
independent. Let $x,y,z\in G$ be arbitrary. We compute $f(zxy)$ by
using the associativity of the operation of $G$. We have
\begin{equation}\label{Eq.5-3}f(z(xy))=f(z)h_{1}(xy)+\mu(z)h_{2}(xy)+\chi(z)A(z)h(xy).\end{equation}
On the other hand, by using (\ref{Eq.5-2}), that $\mu$ and $\chi$
are multiplicative and $A$ is additive, we get that
\begin{equation}\label{Eq.5-4}\begin{split}&f((zx)y)=f(zx)h_{1}(y)+\mu(zx)h_{2}(y)+\chi(zx)A(zx)h(y)\\
&=[f(z)h_{1}(x)+\mu(z)h_{2}(x)+\chi(z)A(z)h(x)]h_{1}(y)+\mu(zx)h_{2}(y)+\chi(zx)A(zx)h(y)\\
&=f(z)h_{1}(x)h_{1}(y)+\mu(z)[h_{2}(x)h_{1}(y)+\mu(x)h_{2}(y)]+\chi(z)\chi(x)A(x)h(y)\\&+\chi(z)A(z)[h(x)h_{1}(y)+\chi(x)h(y)].\end{split}\end{equation}
Since the set $\{f,\mu,\chi,\chi\,A\}$ is linearly independent and
$x,y,z$ are arbitrary we derive from (\ref{Eq.5-3}) and
(\ref{Eq.5-4}) that
\begin{equation}\label{Eq.5-5}h_{1}(xy)=h_{1}(x)h_{1}(y),\end{equation}
\begin{equation}\label{Eq.5-6}h_{2}(xy)=h_{2}(x)h_{1}(y)+\mu(x)h_{2}(y),\end{equation}
\begin{equation}\label{Eq.5-7}\chi(x)A(x)h(y)=0\end{equation}
and
\begin{equation}\label{Eq.5-8}h(xy)=h(x)h_{1}(y)+\chi(x)h(y),\end{equation}
for all $x,y\in G$. The functional equation (\ref{Eq.5-5}) says that
$\mu_{1}:=h_{1}$ is a multiplicative function on $G$. If there
exists $x_{0}\in G$ such that $\mu_{1}(x_{0})=0$ then $\mu_{1}=0$.
So, by putting $y=e$ in (\ref{Eq.5-2}), we get that
$f=h_{2}(e)\mu+h(e)\chi\,A$, which contradicts that the set
$\{f,\mu,\chi,\chi\,A\}$ is linearly independent. Thus $\mu_{1}$ is
a character of $G$.
\par Since $\chi$ is a character of $G$ and $A$ is an additive function on $G$
we derive from (\ref{Eq.5-7}) that $h=0$.
\par Taking into account that $h_{1}=\mu_{1}$ the functional
equation (\ref{Eq.5-6}) becomes
$$h_{2}(xy)=h_{2}(x)\mu_{1}(y)+\mu(x)h_{2}(y)$$ for all $x,y\in G$,
i.e., $h_{2}=\varphi_{\mu_{1},\mu}$.
\par To find $f$ we put $x=e$ in (\ref{Eq.5-2}) which yields that
$$f=a\mu_{1}+\varphi_{\mu_{1},\mu},$$ where $a:=f(e)$. The result occurs in part
(2).
\par Case 2. Suppose that the set $\{f,\mu,\chi,\chi\,A\}$ is linearly
dependent. Since $\{\mu,\chi,\chi\,A\}$ is linearly independent and
$f\neq0$ there exists a triple
$(a,b,c)\in\mathbb{C}^{3}\setminus\{(0,0,0)\}$ such that
\begin{equation}\label{Eq.5-9}f=a\mu+b\chi+c\chi\,A.\end{equation} Combining
(\ref{Eq.5-2}) and (\ref{Eq.5-9}) we get by a small computation that
\begin{equation}\label{Eq.5-20}f(xy)=(ah_{1}(y)+h_{2}(y))\mu(x)+bh_{1}(y)\chi(x)+(ch_{1}(y)+h(y))\chi(x)A(x),\end{equation}
for all $x,y\in G.$\\On the other hand (\ref{Eq.5-9}) gives
\begin{equation}\label{Eq.5-21}f(xy)=a\mu(y)\mu(x)+(b\chi(y)+c\chi(y)A(y))\chi(x)+c\chi(y)\chi(x)A(x),\end{equation}
for all $x,y\in G$. Since the set $\{\mu,\chi,\chi\,A\}$ is linearly
independent we get from (\ref{Eq.5-20}) and (\ref{Eq.5-21}) that
\begin{equation}\label{Eq.5-22}ah_{1}+h_{2}=a\mu,\end{equation}
\begin{equation}\label{Eq.5-23}bh_{1}=b\chi+c\chi\,A\end{equation}
and
\begin{equation}\label{Eq.5-24}ch_{1}+h=c\chi.\end{equation}
\par If $b=0$ then we get from (\ref{Eq.5-23}) that $c=0$. So,
(\ref{Eq.5-9}) and (\ref{Eq.5-24}) imply that $f=a\mu$ and $h=0$.
From (\ref{Eq.5-22}) we obtain $h_{2}=a\mu-ah_{1}$ with $h_{1}$
arbitrary. The result occurs in part (1).
\par If $b\neq0$ then we derive from (\ref{Eq.5-22}), (\ref{Eq.5-23})
and (\ref{Eq.5-24}) that
$$h_{1}=\chi+\dfrac{c}{b}\chi\,A,\,\,h_{2}=a\mu-a\chi-\dfrac{ac}{b}\chi\,A,\,\,h=-\dfrac{c^{2}}{b}\chi\,A.$$
In (\ref{Eq.5-9}) and the identities above we put
$\alpha=\dfrac{c}{b}$, which yields that
$$f=a\mu+b\chi+b\alpha\chi\,A,\,\,h_{1}=\chi+\alpha\chi\,A,\,\,h_{2}=a\mu-a\chi-a\alpha\chi\,A,\,\,h=-\alpha^{2} b\chi\,A.$$
The result occurs in part (3).
\end{proof}
\subsection{Solutions of Eq. (\ref{Eq.5-1})}\quad\\\\
In Theorem \ref{thm7-2} we solve the functional equation
(\ref{Eq.5-1}), i.e.,
\begin{equation*}f(xy)=g_{1}(x)h_{1}(y)+\mu(x)h_{2}(y)+\chi(x)A(x)h(y),\,x,y\in G,\end{equation*}
where $f,g_{1},h_{1},h_{2},h$ are the unknown functions to be
determined.
\begin{thm}\label{thm7-2} The solutions $(f,g_{1},h_{1},h_{2},h)$ of
(\ref{Eq.5-1}) are:\\
(1) $f=a\mu$, $g_{1}$ arbitrary, $h_{1}=0$, $h_{2}=a\mu$ and
$h=0$, where $a$ ranges over $\mathbb{C}$;\\
(2) $f=a\mu$, $g_{1}=-b\mu+c\chi\,A$, $h_{1}\neq0$ arbitrary,
$h_{2}=a\mu+bh_{1}$ and $h=-ch_{1}$, where $a$, $b$ and $c$ range
over $\mathbb{C}$;\\
(3)
$f=a\mu+b\chi\,A,\,g_{1}=\alpha\mu+\beta\chi+\gamma\chi\,A,\,h_{1}=c\chi\,A,\,h_{2}=a\mu-\alpha
c\chi\,A$ and $h=b\chi-\gamma c\chi\,A$, where
$a,\,\alpha,\,\beta,\,\gamma,\,b$ and $c$ are constants such that $\beta bc\neq0$ and $b=\beta c$;\\
(4)
$f=a_{1}\mu+\varphi_{\mu_{1},\mu},\,g_{1}=(\varphi_{\mu_{1},\mu}-a_{2}\mu)/a+a_{3}\chi\,A,\,h_{1}=a\mu_{1},\,h_{2}=a_{1}\mu+a_{2}\mu_{1}+\varphi_{\mu_{1},\mu}$\\$\quad\text{and}\quad
h=-aa_{3}\mu_{1}$, where $\mu_{1}$ is a character of $G$,
$a_{1},\,a_{2}$ and $a_{3}$ range
over $\mathbb{C}$ and $a$ over $\mathbb{C}\setminus\{0\}$;\\
(5) $f=a\mu+ca_{2}\chi+\alpha
ca_{2}\chi\,A,\,g_{1}=a_{1}\mu+a_{2}\chi+a_{3}\chi\,A,\,h_{1}=c\chi+\alpha
c\chi\,A,
\\h_{2}=a\mu-ca_{1}\chi-\alpha ca_{1}\chi\,A\quad\text{and}\quad
h=c(\alpha a_{2}-a_{3})\chi-\alpha ca_{3}\chi\,A$, where
$\alpha,\,a,\,a_{1}$ and $a_{3}$ range over $\mathbb{C}$, and $c$
and $a_{2}$ over $\mathbb{C}\setminus\{0\}$.
\end{thm}
\begin{proof}  We check by elementary computations that if $f,g_{1},h_{1},h_{2}$ and $h$
are of the forms (1)-(5) then $(f,g_{1},h_{1},h_{2},h)$ is a
solution of (\ref{Eq.5-1}),
so left is that any solution $(f,g_{1},h_{1},h_{2},h)$ of (\ref{Eq.5-1}) fits into (1)-(5).\\
There are two cases to consider.
\par Case 1. $f$ is proportional to $\mu$.
By Proposition \ref{prop55} we get solutions (1) and (2).
\par Case 2. $f$ is not proportional to $\mu$. By putting $y=e$ in
the functional equation (\ref{Eq.5-1}) we get that
\begin{equation}\label{Eq.5-25}f=h_{1}(e)g_{1}+h_{2}(e)\mu+h(e)\chi\,A.\end{equation}
We consider two subcases according to whether $h_{1}(e)=0$ or not.
\par Subcase 2.1. $h_{1}(e)=0$. Then (\ref{Eq.5-25}) becomes
\begin{equation}\label{Eq.5-26}f=a\mu+b\chi\,A,\end{equation}
where $a:=h_{2}(e)$ and $b:=h(e)$. So the set
$\{f,\mu,\chi,\chi\,A\}$ is linearly dependent. As $f$ is not
proportional to $\mu$ we derive, according to Lemma \ref{lem53}(1),
that $h_{1}\neq0$. So, by Lemma \ref{lem53}(2), the set
$\{g_{1},\mu,\chi,\chi\,A\}$ is linearly dependent. As the set
$\{\mu,\chi,\chi\,A\}$ is linearly independent we deduce that there
exists a triple $(\alpha,\beta,\gamma)\in\mathbb{C}^{3}$ such that
\begin{equation}\label{Eq.5-27}g_{1}=\alpha\mu+\beta\chi+\gamma\chi\,A.\end{equation}
Let $y\in G$ be arbitrary. By substituting (\ref{Eq.5-26}) and
(\ref{Eq.5-27}) in (\ref{Eq.5-1}) we obtain
\begin{equation*}\begin{split}&a\mu(x)\mu(y)+b\chi(x)\chi(y)A(y)+b\chi(x)A(x)\chi(y)\\
&=(\alpha\mu(x)+\beta\chi(x)+\gamma\chi(x)A(x))h_{1}(y)+\mu(x)h_{2}(y)+\chi(x)A(x)h(y)\\
&=\mu(x)(\alpha h_{1}(y)+h_{2}(y))+\chi(x)(\beta
h_{1}(y))+\chi(x)A(x)(\gamma
h_{1}(y)+h(y)),\end{split}\end{equation*} for all $x\in G$. Since
the set $\{\mu,\chi,\chi\,A\}$ is linearly independent we derive
from the identity above that $\alpha h_{1}(y)+h_{2}(y)=a\mu(y)$,
$\beta h_{1}(y)=b\chi(y)A(y)$ and $\gamma h_{1}(y)+h(y)=b\chi(y)$.
So, $y$ being arbitrary, we deduce that
\begin{equation}\label{Eq.5-28}\alpha h_{1}+h_{2}=a\mu,\,\beta h_{1}=b\chi\,A\quad\text{and}\quad\gamma h_{1}+h=b\chi.\end{equation}
Since $f$ is not proportional to $\mu$ we get from (\ref{Eq.5-26})
that $b\neq0$. Hence, the second identity in (\ref{Eq.5-28}) implies
that $\beta\neq0$ because $\chi\,A\neq0$. So, by putting
$c=\dfrac{b}{\beta}$ we derive from (\ref{Eq.5-28}) that
\begin{equation*}h_{1}=c\chi\,A,\,h_{2}=a\mu-\alpha c\chi\,A\quad\text{and}\quad h=b\chi-\gamma c\chi\,A.\end{equation*}
The result occurs in part (3).
\par Subcase 2.2. $h_{1}(e)\neq0$. Here we get from (\ref{Eq.5-25})
that
\begin{equation}\label{Eq.5-29}g_{1}=\dfrac{1}{h_{1}(e)}f-\dfrac{h_{2}(e)}{h_{1}(e)}\mu-\dfrac{h(e)}{h_{1}(e)}\chi\,A.\end{equation}
Substituting this back into (\ref{Eq.5-1}) we get, by elementary
computation, that
$$f(xy)=f(x)\dfrac{h_{1}(y)}{h_{1}(e)}+\mu(x)(h_{2}(y)-\dfrac{h_{2}(e)}{h_{1}(e)}h_{1}(y))+\chi(x)A(x)(h(y)-\dfrac{h(e)}{h_{1}(e)}h_{1}(y)),$$
for all $x,y\in G$. Then the quadruple
$(f,h_{1}/h_{1}(e),h_{2}-\dfrac{h_{2}(e)}{h_{1}(e)}h_{1},h-\dfrac{h(e)}{h_{1}(e)}h_{1})$
satisfies the functional equation (\ref{Eq.5-2}), hence, according
to Theorem \ref{thm7-1} and seeing that $f$ is not proportional to
$\mu$, that quadruple falls into two categories:\\
(i) $f=a_{1}\mu_{1}+\varphi_{\mu_{1},\mu}$,
$h_{1}/h_{1}(e)=\mu_{1}$,
$h_{2}-\dfrac{h_{2}(e)}{h_{1}(e)}h_{1}=\varphi_{\mu_{1},\mu}$ and
$h-\dfrac{h(e)}{h_{1}(e)}h_{1}=0$, where $\mu_{1}$ is a character of
$G$ and $a_{1}$ ranges over $\mathbb{C}$. Then
\begin{equation*}f=a_{1}\mu_{1}+\varphi_{\mu_{1},\mu},\,h_{1}=h_{1}(e)\mu_{1},\,h_{2}=h_{2}(e)\mu_{1}+\varphi_{\mu_{1},\mu}\quad\text{and}\quad h=h(e)\mu_{1}.\end{equation*}
So, (\ref{Eq.5-29}) gives
\begin{equation*}g_{1}=(a_{1}\mu_{1}+\varphi_{\mu_{1},\mu}-h_{2}(e)\mu)/h_{1}(e)-\dfrac{h(e)}{h_{1}(e)}\chi\,A.\end{equation*}
Defining
$$a:=h_{1}(e)\neq0,\,a_{2}:=h_{2}(e)-a_{1},\,\quad\text{and}\quad a_{3}:=-\dfrac{h(e)}{h_{1}(e)}$$
and written $\varphi_{\mu_{1},\mu}$ instead of
$a_{1}(\mu_{1}-\mu)+\varphi_{\mu_{1},\mu}$ the formulas above read
\begin{equation*}\begin{split}&f=a_{1}\mu+\varphi_{\mu_{1},\mu},\,g_{1}=(\varphi_{\mu_{1},\mu}-a_{2}\mu)/a+a_{3}\chi\,A,\,h_{1}=a\mu_{1},\\
&h_{2}=a_{1}\mu+a_{2}\mu_{1}+\varphi_{\mu_{1},\mu}\quad\text{and}\quad
h=-aa_{3}\mu_{1}.\end{split}\end{equation*} Moreover since $f$ is
not proportional to $\mu$, $\varphi_{\mu_{1},\mu}$ is non zero. The
result obtained occurs in part (4).\\
(ii) $f=a\mu+b\chi+\alpha b\chi\,A$,
$h_{1}/h_{1}(e)=\chi+\alpha\chi\,A$,
$h_{2}-\dfrac{h_{2}(e)}{h_{1}(e)}h_{1}=a\mu-a\chi-\alpha a\chi\,A$
and $h-\dfrac{h(e)}{h_{1}(e)}h_{1}=-\alpha^{2} b\chi\,A$, where $a$
and $\alpha$ range over $\mathbb{C}$, and $b$ over
$\mathbb{C}\setminus\{0\}$. Then
\begin{equation}\label{Eq.5-30}\begin{split}&h_{1}=h_{1}(e)\chi+\alpha
h_{1}(e)\chi\,A,\,h_{2}=a\mu+(h_{2}(e)-a)\chi+\alpha(h_{2}(e)-a)\chi\,A\\&\text{and}\quad
h=h(e)\chi+\alpha(h(e)-\alpha b)\chi\,A.\end{split}\end{equation}
So, by using (\ref{Eq.5-29}), we get that
\begin{equation}\label{Eq.5-31}g_{1}=\dfrac{a-h_{2}(e)}{h_{1}(e)}\mu+\dfrac{b}{h_{1}(e)}\chi+\dfrac{\alpha b-h(e)}{h_{1}(e)}\chi\,A. \end{equation}
Defining $c:=h_{1}(e)\neq0$, $a_{1}:=\dfrac{a-h_{2}(e)}{h_{1}(e)}$,
$a_{2}:=\dfrac{b}{h_{1}(e)}\neq0$ and $a_{3}:=\dfrac{\alpha
b-h(e)}{h_{1}(e)}$, we get, by using the formula of $f$ above,
(\ref{Eq.5-30}) and (\ref{Eq.5-31}), that
\begin{equation*}\begin{split}f&=a\mu+ca_{2}\chi+\alpha ca_{2}\chi\,A,\,g_{1}=a_{1}\mu+a_{2}\chi+a_{3}\chi\,A,\,h_{1}=c\chi+\alpha c\chi\,A,
\\&h_{2}=a\mu-ca_{1}\chi-\alpha ca_{1}\chi\,A\quad\text{and}\quad
h=c(\alpha a_{2}-a_{3})\chi-\alpha
ca_{3}\chi\,A.\end{split}\end{equation*} The result occurs in part
(5).
\end{proof}
\begin{exple}
If we choice $a_{3}=0$ in part (4) of Theorem \ref{thm7-2} we get
that $h=0$ and
\begin{equation*}f=a_{1}\mu+\varphi_{\mu_{1},\mu},\,g_{1}=(\varphi_{\mu_{1},\mu}-a_{2}\mu)/a,\,h_{1}=a\mu_{1}
\quad\text{and}\quad
h_{2}=a_{1}\mu+a_{2}\mu_{1}+\varphi_{\mu_{1},\mu},\end{equation*}
with the same constraints on
$\mu_{1},\,\varphi_{\mu_{1},\mu},\,a,\,a_{1}$ and $a_{2}$, which is
the solution obtained in \cite[Theorem 14]{Stetkaer2}.
\end{exple}
\begin{exple}
\par If we choice $\alpha=a_{3}=0$ in part (5) we get that $h=0$ and
\begin{equation*}f=a\mu+ca_{2}\chi,\,g_{1}=a_{1}\mu+a_{2}\chi,\,h_{1}=c\chi\quad\text{and}\quad h_{2}=a\mu-ca_{1}.\end{equation*}
With the notations
$c_{1}:=a+cc_{2},\,c_{2}:=-ca_{1}-ca_{2},\,\lambda:=ca_{2}\neq0$ and
$\widetilde{f}:=\lambda(\chi-\mu)$ the formulas above read
\begin{equation*}f=c_{1}\mu+\widetilde{f},\,g_{1}=(\widetilde{f}-c_{2}\mu)/c,\,h_{1}=c\chi
\quad\text{and}\quad
h_{2}=c_{1}\mu+c_{2}\chi+\widetilde{f},\end{equation*} which is a
solution obtained in \cite[Theorem 14]{Stetkaer2} corresponding to
the non zero central solutions $\widetilde{f}=\lambda(\chi-\mu)$ of
the functional equation
$$\widetilde{f}(xy)=\widetilde{f}(x)\chi(y)+\mu(x)\widetilde{f}(y),\,x,y\in G.$$
\end{exple}
\par We close with non-examples. On finite monoids or groups the
only additive function is $A=0$, so the functional equations
(\ref{equ51}) and (\ref{Eq.5-1}) become
\begin{equation*}f(xy)=g_{1}(x)h_{1}(y)+g(x)h_{2}(y),\quad x,y\in
G,\end{equation*}where $G$ is a monoid, $g$ is the average of two
distinct nonzero multiplicative functions $\mu_{1}$ and $\mu_{2}$ on
$G$, and
\begin{equation}f(xy)=g_{1}(x)h_{1}(y)+\mu(x)h_{2}(y),\quad x,y\in
G,\end{equation} where $G$ is a group and $\mu$ is a character of
$G$, which was recently solved in \cite{Ebanks and Stetkaer} and
\cite{Stetkaer1}.
\par on the other hand, on perfect groups, i.e. groups $G$ such that
$G=[G,G]$, the only character is $\chi=1$ so that the function
$g=\dfrac{\mu_{1}+\mu_{2}}{2}$ cannot exist. As examples of perfect
groups we cite connected groups and semi-simple Lie groups like
$SO(n)$, $SL(n,\mathbb{R})$ and $SL(n,\mathbb{C})$ for $n\geq2$.


\end{document}